\documentclass[11pt]{article}
\usepackage[utf8]{inputenc}
\usepackage{graphicx}
\usepackage{latexsym}
\usepackage{amsmath}
\usepackage{amssymb}
\usepackage{amsthm}
\usepackage{hyperref}
\usepackage{multicol}
\usepackage{subcaption}
\usepackage{multirow}
\usepackage{caption,bm}
\usepackage[inline]{enumitem}
\newcommand{\Mod}[1]{\ (\mathrm{mod}\ #1)}

\makeatother

\def\@fnsymbol#1{\ensuremath{\ifcase#1\or *\or \dagger\or \ddagger\or
		\mathsection\or \mathparagraph\or \|\or **\or \dagger\dagger
		\or \ddagger\ddagger \else\@ctrerr\fi}}
\newcommand{\ssymbol}[1]{^{\@fnsymbol{#1}}}

\DeclareMathOperator{\cir}{Circ}
\DeclareMathOperator{\In}{In}

\numberwithin{equation}{section}
\theoremstyle{plain}

\newtheorem{theorem}{Theorem}[section]
\newtheorem{lemma}[theorem]{Lemma}

\newtheorem{remark}[theorem]{Remark}
\newtheorem{proposition}[theorem]{Proposition}
\newtheorem{definition}[theorem]{Definition}
\newtheorem{example}[theorem]{Example}
\usepackage{mathtools}
\usepackage{pgf,tikz}
\usepackage{tikz}
\usepackage{tkz-graph}
\usepackage{array}
\tikzstyle{vertex}=[circle,draw, inner sep=0pt, minimum size=4pt] 

\usetikzlibrary{decorations.markings}
\usepackage{lipsum}

\usepackage[most]{tcolorbox}
\allowdisplaybreaks
\title{On Eccentricity Matrices of Wheel Graphs}

\author{I. Jeyaraman\textsuperscript{1} and T. Divyadevi\textsuperscript{2} \\\\
	Department of Mathematics\\
	National Institute of 
	Technology		Tiruchirappalli-620 015,  India\\ (e-mail: 
jeyaraman@nitt.edu\textsuperscript{1} and 	
tdivyadevi@gmail.com\textsuperscript{2})} 
\date{}
\begin{document}
\maketitle

\begin{abstract}
	The eccentricity matrix $E(G)$ of a simple connected graph 
$G$ is obtained from the distance matrix 	$D(G)$ of $G$ 
by retaining the largest distance in each row and column, and by defining 
the remaining entries to be zero. This paper focuses on the 
eccentricity 
matrix $E(W_n)$ of the wheel graph $W_n$  with $n$ 
vertices. By establishing a  formula for the determinant of $E(W_n)$, we 
show that $E(W_n)$ is invertible if and only if  $n \not\equiv 
1\Mod3$. We derive a formula for the inverse of $E(W_n)$ by finding a 
vector $\mathbf{w}\in 
\mathbb{R}^n$ and an $n \times n$ symmetric Laplacian-like matrix 
$\widetilde{L}$ 
of rank $n-1$ such that 
\begin{eqnarray*}
	E(W_n)^{-1} =	-\frac{1}{2}\widetilde{L} +
	\frac{6}{n-1}\mathbf{w}\mathbf{w^{\prime}}.
\end{eqnarray*}
Further, we prove an analogous result for the Moore-Penrose inverse of 
$E(W_n)$ 
for the singular case. We also determine the inertia of $E(W_n)$.
\end{abstract}


\section{Introduction}

	Let $G$ be a simple connected  graph 
on $n$ vertices with the vertex set $\{v_1,v_2,\cdots,v_n\}$. Corresponding 
to $G$, several (graph) matrices have been introduced and studied, namely, 
the incidence matrix, the adjacency matrix, the Laplacian matrix, the 
distance matrix, the resistance matrix etc., see \cite{Bapat, Graph spectra 
	book London ms}. We recall the adjacency and distance matrices which are 
relevant to the discussion here. The \textit{adjacency matrix} 
$A(G):=(a_{ij})$ of $G$ is an $n \times n$ matrix with $a_{ij}=1$ if
$v_i$ and $v_j$ are adjacent and zero elsewhere. Let $d(v_i,v_j)$ be the 
length of a shortest path between $v_i$ and $v_j$. The \textit{distance 
	matrix}  of $G$, denoted by $D(G):=(d_{ij})$, is an $n \times n$ matrix 
with 
$d_{ij}=d(v_i,v_j)$  for all $i$ and $j$. Then $A(G)$ and $D(G)$ are 
symmetric matrices with diagonal entries all equal to zero. These matrices 
have been well 
studied in the literature and  have a wide 
range of applications in chemistry, physics, computer science, etc., see  
\cite{Bapat, Brouwer, Graph spectra book London ms} and the references 
therein. Note that the 	adjacency matrix can be derived 
from the 
distance matrix by keeping the smallest non-zero distance (which is equal 	
to one) in each row and column and by setting the 	remaining entries 
to be 
zero.

Randi\'{c} \cite{Randic} introduced a new graph matrix known as 
$D_{\text{MAX}}$ which is obtained from the distance 
matrix $D(G)$ by retaining the largest distance in each row and each 
column of $D(G)$, 
and setting the rest of the entries to zero.
Wang 
et al. \cite{Wang1} introduced the $D_{\text{MAX}}$ matrix in a 
slightly different form using the notion of the eccentricity of a vertex 
and called it  the eccentricity matrix $E(G)$.
The \textit{eccentricity} of a vertex $v_i$, denoted by $e(v_i)$, is defined 
by
$e(v_i)$= max $\{d(v_i,v_j): 1\leq j \leq n\}$.
Let us recall the equivalent formulation 
of $D_{\text{MAX}}$ given in 
\cite{Wang1}. The \textit{eccentricity matrix} of $G$ is an $n \times n$ 
matrix with
\begin{eqnarray*}\label{Eccentricty matrix defintion}
	(E (G))_{ij}= \begin{cases}  d(v_i,v_j)   & \text{if } 
		d(v_i,v_j)=\text{min}\{e(v_i),e(v_j)\},\\
		0 & \text{otherwise},
	\end{cases}
\end{eqnarray*}
where $(E (G))_{ij}$ is the $(i,j)$-th entry of $E(G)$. It is 
also
referred to as the anti-adjacency matrix \cite{Wang1}. Note that $E(G)$ is a 
real symmetric matrix whose entries are non-negative. This new matrix $E(G)$ 
is 
applied to study the boiling point of  hydrocarbons \cite{Wang3}. For more 
details 
on the applications of $E(G)$ in terms of molecular descriptor, we refer to	
\cite{Wang1, Wang3, Randic}. Motivated by the 
concepts and results of other graph matrices, 
several 
spectral properties have been studied for the eccentricity matrix in 
\cite{Wang1,Wang2, Wang3, rajes1, Rajesh, Randic}. The 
relation between 
the eigenvalues of 
$A(G)$ and $E(G)$ has been investigated for certain graphs in 		
\cite{Wang1}. 

Unlike the adjacency and distance matrices of a connected 
graph, the eccentricity matrix is not irreducible \cite{Wang1}. The 
problem of characterizing the irreducible eccentricity matrix posed by 
Wang et al. \cite{Wang1} remains open. They showed 
that the
eccentricity matrix of a tree is irreducible.  Mahato et al. 
\cite{rajes1} gave an alternative proof. In \cite{Wang2}, this result 
was generalized by providing a class of graphs whose eccentricity 
matrices are irreducible. In this paper, we study the eccentricity 
matrix of the wheel graph $W_n$ with $n$ vertices. We show that 
$E(W_n)$ is irreducible.

An important problem in graph matrices is to find the determinants and inertias 
of these matrix classes. A 	famous result in the theory of distance matrix 
states that the determinant 
of a tree $T$ with $n$ vertices is $(-1)^{n-1}(n-1)2^{n-2}$ 
which depends only on the number of vertices \cite{Graham and Pollak}. The 
inertia of 
$D(T)$ was 	also found in \cite{Graham and Pollak}. These results were 
extended to the 
distance matrix of a weighted tree by Bapat et al. \cite{Bapat2}. 
The determinant and inertia of $D(W_n)$ have been computed in \cite{Wheel 
	related graph} and  
the inertias of the eccentricity matrices of certain graphs 
have been investigated in 
\cite{rajes1}. In this paper, we determine a formula to find the 
determinant of $E(W_n)$ and derive the inertia of $E(W_n)$. We show that 
the determinant of $E(W_n)$ is 	$2^{n-2}(1-n)$ if $n \not\equiv1\Mod3$,
which depends only on the number of vertices, and zero if $n\equiv1\Mod3$, 
see Theorem \ref{wheel determinant}. 

The problem of finding the {inverses} of different graph 
matrices has 
been extensively studied in the literature, see \cite{Balaji,Bapat,Bapat 
	Inverse of distance block graph, Graham, helm graph, Zhou}. 
In 
order to motivate 
our next result, we recall the Laplacian matrix $L$ of $G$. The 
\textit{Laplacian matrix} of $G$ is given by $L:=\bar{D}-A(G)$, where 
$\bar{D}$ is the diagonal matrix whose $i$-th  
diagonal 
entry is deg($v_i$), the degree of the vertex $v_i$.  Then $L$ is an $n 
\times n$ symmetric matrix whose row sums are zero and rank of $L$ is 
$n-1$, see \cite{Bapat}.
Graham and Lov\'{a}sz \cite{Graham} gave an interesting formula for the 
inverse of the distance matrix $D(T)$ of a tree $T$ which is given by 
\begin{eqnarray}\label{Lovaz}
	D(T)^{-1}=-\frac{1}{2}L+\frac{1}{2(n-1)}\bm{\tau\tau^{\prime}},
\end{eqnarray}
where $\bm{\tau}^{\prime}=\big(2-\text{deg}(v_1), 2-\text{deg}(v_2), \cdots, 
2-\text{deg}(v_n)\big)$. This result was extended to the distance 
matrix of a weighted tree in \cite{Bapat2} and to the resistance matrix of $G$, 
see \cite{Bapat}.

Let us recall that a real square matrix $\widetilde{L}$ is called a 
Laplacian-like matrix \cite{Zhou} if $\widetilde{L}\mathbf{	e}=0$ and 
$\mathbf{{e}^{\prime}}\widetilde{L}=0$, where $\mathbf{	e}$ denotes the column 
vector whose entries are all one. Inspired by the result of 
Graham and Lov\'{a}sz \cite{Graham}, {similar inverse formulae 
	have been 
	obtained for the distance matrices} of wheel 
graph $W_n$ when $n$ is even
\cite{Balaji},  helm graphs \cite{helm graph}, cycles 
\cite{Bapat2}, block graphs \cite{Bapat Inverse of 
	distance block graph} and complete 
graphs \cite{Zhou}. For all these matrices, the inverse formula is 
expressed as the 
sum of a Laplacian-like matrix and a rank one matrix. Almost all the distance 
matrices of the above mentioned graphs are Euclidean distance matrix (EDM), see 
Section $2.3$ for the definition. The inverse formula has been studied for EDM 
in \cite{Balaji EDM} where the result of  Graham and Lov\'{a}sz 
\cite{Graham} was generalized. It would be interesting and challenging to study 
the inverse formula for non-EDM. Note that $E(W_n)$ is not an EDM (see Section 
$2.3$). Therefore studying $E(W_n)$ will give new ideas and results in the 
class of irreducible non-EDM. Motivated by the inverse formulae of the distance 
matrices of several graphs, we find an inverse formula for $E(W_n)$ when $n 
\not\equiv1\Mod3$ which is of 
the form (\ref{Lovaz}), see Theorem \ref{inverse 
	formula n is even 0 mod}.

Let us turn our attention to the Moore-Penrose (generalized) inverse of graph 
matrices. Let $A$ be an $m \times n$ real matrix. An $n \times m$ matrix $X$ 
is 
called the \textit{Moore-Penrose inverse} of $A$ if
$AXA=A,~ XAX=X$,~ $(AX)^{\prime}=AX$ and $(XA)^{\prime}=XA$. It is known that 
Moore–Penrose inverse, denoted by 
$A\ssymbol{2}$,  exists and is unique. Further, if $A$ is a non-singular 
matrix, 
then $A\ssymbol{2}$ coincides with the usual inverse 
$A^{-1}$. For more details, we refer to
\cite{Moore inverse book}. 

The Moore-Penrose inverse has been presented for the incidence matrices of 
complete multipartite graph, bi-block graph, distance regular graph, tree 
etc., and for EDM, see \cite{MP inverse, Bapat, Balaji EDM} and references 
therein. In a 
recent paper,
Balaji et al. 
\cite{Balaji odd 
	wheel graph} provided a formula 
to compute the Moore-Penrose inverse of the
distance matrix of a wheel graph with an odd number of vertices, which is 
given by
\begin{equation}\label{Balaji MP formula}
	D(W_n)\ssymbol{2}=
	-\frac{1}{2}{\bar{L}}+\frac{4}{n-1}\mathbf{w}\mathbf{w^{\prime}},
\end{equation}
where $\bar{L}$ is a real symmetric Laplacian-like matrix of order $n$ and 
$\mathbf{w} \in 
\mathbb{R}^{n}$. 
It is natural 
to ask whether an analogous result can be studied for the eccentricity matrix 
of $W_n$ for the singular case. Answering this question positively, in this 
paper, we 
obtain a  formula for  $E(W_n)\ssymbol{2}$ when $n 
\equiv1\Mod3$ in the form 
of $(\ref{Lovaz})$ by introducing a symmetric Laplacian-like matrix 
$\widetilde{L}$ (Theorem \ref{MP formula}).

The article is organized as follows. In Section $2$, we fix the 
notations, collect some known results and study some 
properties of the eccentricity matrix of the wheel graph $E(W_n)$. In 
Section 3, we establish a formula to compute the determinant of $E(W_n)$. 
As a  consequence, we show that $E(W_n)$ is invertible if and only if $n 
\not\equiv1\Mod3$.  Section 4 deals with the inertia 
of 	$E(W_n)$. As a by-product, we obtain the determinant and the 
inertia of the 
eccentricity matrix of an  edge-deleted subgraph of $W_n$ where the edge 
lies 
on the cycle. In  Section 5, we obtain an explicit formula similar to 
(\ref{Lovaz}) for the inverse of $E(W_n)$ when $n \not\equiv1\Mod3$. 
The inverse formula is expressed as the sum of a symmetric Laplacian-like 
matrix $\widetilde{L}$ and a rank one matrix which is  given by 
	$	E(W_n)^{-1}=
	-\frac{1}{2}\widetilde{L}+\frac{6}{n-1}\mathbf{w}\mathbf{w^{\prime}},$	
where $\mathbf{w}=\frac{1}{6}(7-n,1,\cdots,1)^{\prime} \in \mathbb{R}^{n}$ 
and the rank of  
$\widetilde{L}$ is one less than that of $E(W_n)$. 
In Section 6, 
an analogous result is presented for the 
Moore-Penrose inverse of $E(W_n)$ when $n \equiv1\Mod3$ by constructing a 
symmetric Laplacian-like matrix $\hat{L}$ of rank $n-3$.

	\section{Preliminaries and  Properties of $E(W_n)$}\label{sec 2}
In this section, we first introduce a few notations and present some 
preliminary results on circulant matrices that are needed in the paper. 	
Next, we compute the eccentricity matrix of the wheel graph and
study some of its properties. We conclude this section by finding 
its spectral radius.

We assume that all the vectors are column vectors  and are denoted by 
lowercase boldface letters. {As usual $I_n$ is the identity 
	matrix of order $n$ and 
	$\mathbf{e_i}$ is the vector in $\mathbb{R}^n$ with the
	$i$-th coordinate $1$ and $0$ elsewhere. We write 
	$\mathbf{e}_{n \times 
		1}$ ($\mathbf{0}_{n \times 1}$) to represent the vector in 
		$\mathbb{R}^n$ 
	whose coordinates are all one (respectively, zero)}. We use the notation 
$J_{n \times m} (0_{n \times m})$ to denote the matrix with 
all elements equal to $1$ (respectively, $0$) and we simply  write
$J_n$ (respectively, $0_n)$ if $n=m$. 
We omit the subscript if the {orders} of 
the vector and the matrix are clear from the context. For 
the matrix $A$, 
we denote the	
transpose of $A$, the null space of $A$, the range of $A$, $i$-th row of 
$A$ and  $j$-th column 
of $A$ by $A^{\prime}$, $N(A)$, $R(A)$,  $A_{i*}$ and $A_{*j}$, 
respectively. 
The notation $A>0 $ ($A \geq 0$) means that all the 
entries of $A$ are positive (respectively, non-negative). The 
determinant of a square matrix $A$ is denoted by $\det(A)$.

For a vector $\mathbf{c} =(c_1, \cdots, c_{n})^{\prime}$ in 
$\mathbb{R}^{n}$, the 
notation 
$\cir(\mathbf{c^{\prime}})$ stands for the circulant matrix of 
order 
$n$, and $T_n{(a,b,c)}$ 
denotes the 
tridiagonal matrix of order $n$ which are given by 
\begin{equation*}\label{C}
	\cir(\mathbf{c}^{\prime}) = 
	\begin{bmatrix}
		c_1&c_2&c_3&\cdots&c_{n-1}&c_{n}\\
		c_{n}&c_1&c_2&\cdots&c_{n-2}&c_{n-1}\\
		c_{n-1}&c_{n}&c_1&\cdots&c_{n-3}&c_{n-2}\\
		\vdots&\vdots&&\ddots&\vdots&\vdots\\
		c_3&c_4&c_5&\cdots&c_1&c_2\\
		c_2&c_3&c_4&\cdots&c_{n}&c_1
	\end{bmatrix} ~\text{and}~ T_n{(a,b,c)} =  \begin{bmatrix}
		a& b& 0&  \cdots & 0& 0& 0\\
		c& a& b&  & 0& 0& 0 \\
		0& c& a&  & 0& 0& 0\\
		\vdots&& \ddots&\ddots&\ddots&  \\
		0& 	&&&a&b& 0\\
		0&\cdots&&	&c&a& b\\
		0&\cdots&&	&0&c&a
	\end{bmatrix},
\end{equation*}
respectively where $a,b$ and $c$ are real numbers. The following properties of 
the circulant matrix will be used several times. We refer to the books 
\cite{Horn, Zhang} for more details.
\vspace{-.5cm}
\begin{align}\notag
	\intertext{Let $\mathbf{x,y} \in \mathbb{R}^n$. Suppose that	
		$A=\cir(\mathbf{x^{\prime}})$ and $B=\cir(\mathbf{y^{\prime}})$. 
		Then} \label{commutative property}
	AB&= BA,\\\label{sum of circulant property}
	\cir (a \mathbf{x^{\prime}}+ b \mathbf{y^{\prime}}) &=aA+bB, 
	\intertext{and}\label{product of circulant property}
	AB&=\cir(\mathbf{x^{\prime}}B).
	\intertext{ Let	$T:\mathbb{R}^{n}\rightarrow \mathbb{R}^{n}$ 
		be defined by}\label{right shift operator}
	T\big((f_1,f_2,\cdots,f_{n})^{\prime}\big) &= 
	(f_{n},f_1,\cdots,f_{n-1})^{\prime}.
\end{align} 

\begin{lemma}\label{Image of a vector which  has repeated components}
	Let $T$ be the operator defined in $(\ref{right shift 
		operator})$. Let $n \equiv 0 
	\Mod{3}$ and 
	$C:=\cir(\mathbf{c}^{\prime})$. If 
	$\mathbf{g}^{\prime}=(\underbrace{\alpha,\beta,\gamma},~\underbrace{\alpha,\beta,\gamma},~\cdots,
	~\underbrace{\alpha,\beta,\gamma}) \in \mathbb{R}^{n}$, then
	$ \mathbf{g}^{\prime}C= 
	(\underbrace{\tau_1,\tau_2,\tau_3,}~ 
	\underbrace{\tau_1,\tau_2,\tau_3},~\cdots,~
	\underbrace{\tau_1,\tau_2,\tau_3}), $
	where 
	$ \tau_1=\mathbf{g}^{\prime}C_{*1}, ~
	\tau_2=\big(T^2(\mathbf{g})\big)^{\prime}C_{*1}$~ and 
	$ \tau_3=\big(T(\mathbf{g})\big)^{\prime}C_{*1}$. 	
\end{lemma}
\begin{proof}
	Let $1 \leq k \leq n$. Note that $C_{*k}=(c_k,c_{k-1}, 
	\cdots, c_2,c_1, 
	c_n,c_{n-1}, \cdots, c_{k+1})^{\prime}$ and the $k$-th coordinate of 
	$\mathbf{g}^{\prime}C$ is $ 
	\mathbf{g}^{\prime}C_{*k}$. By the usual matrix 
	multiplication, we see that
	\begin{equation*}
		\mathbf{g}^{\prime}C_{*k}=
		\alpha(c_k+c_{k-3}+\cdots+c_{k+3})+\beta(c_{k-1}+c_{k-4}+\cdots+c_{k+2})+
		\gamma(c_{k-2}+c_{k-5}+\cdots+c_{k+1}).
	\end{equation*}  
	If  $k\equiv1\Mod3$, then $ \mathbf{g}^{\prime}C_{*k}=\alpha 
	\big (\sum_{i\equiv1\Mod3}c_i\big)+ \beta 
	\big(\sum_{i\equiv0\Mod3}c_{i}\big)+   	
	\gamma\big(\sum_{i\equiv2\Mod3}c_{i}\big)$.
	That is, $\mathbf{g}^{\prime}C_{*1}= \mathbf{g}^{\prime}C_{*k}$ for all 
	${k\equiv1\Mod3}$. The remaining two cases ${k\equiv2\Mod3}$  and 
	${k\equiv0\Mod3}$ can be be proved similarly.
\end{proof}

\begin{figure}
	\begin{center}
		\begin{tikzpicture}[scale=1]
			\draw[fill=black] (0,0) circle (2pt);\node at (0,-0.3) {$ 
				v_1 $};
			\draw[fill=black] (1.8,0) circle (2pt);
			\draw[fill=black] (-1.8,0) circle (2pt);
			\draw[fill=black] (1,1) circle (2pt);
			\draw[fill=black] (1,-1) circle (2pt);
			\draw[fill=black] (-1,1) circle (2pt);
			\draw[fill=black] (-1,-1) circle (2pt);
			\draw[thin] 
			(0,0)--(1.8,0)--(1,1)--(-1,1)--(-1,1)--(-1.8,0)--(-1,-1)--(1,-1)--(1.8,0);
			\draw[thin] (0,0)--(-1.8,0);
			\draw[thin] (0,0)--(-1,1);
			\draw[thin] (0,0)--(-1,-1);
			\draw[thin] (0,0)--(1,1);
			\draw[thin] (0,0)--(1,-1);
			\node at (1,1.3) {$ v_2 $};
			\node at (1,-1.3) {$ v_4 $};
			\node at (-1,-1.3) {$ v_5 $};
			\node at (-1,1.3) {$ v_7 $};
			\node at (2.1,0) {$ v_3 $};
			\node at (-2.1,0) {$ v_6 $};
		\end{tikzpicture}
	\end{center}
	\caption{Wheel graph on seven vertices}\label{wheelgraph}
\end{figure}
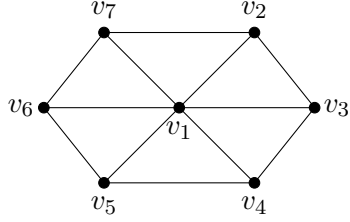
\subsection{  Eccentricity matrix of Wheel Graph}
For $n\geq 4$, the wheel graph on $n$ vertices $W_n$ is a graph 
containing a
cycle of length $n-1$ and a vertex  (called the hub) not in the 
cycle which is adjacent to every other vertex in the cycle. Throughout this 
paper, we label the 
vertices of $W_n$ by $v_1,v_2, \cdots v_n$, where the hub of $W_n$ is 
labelled as		$v_1$ and the vertices in the cycle are labelled as 
$v_2, \cdots 
v_n$ clockwise (see figure \ref{wheelgraph} for $W_7$). 

We observe that $D(W_4)=E(W_4)$. We assume that 
$n\geq 5$. Since $v_1$ is adjacent to all other $v_i'$s in $W_n$, we 
have $e(v_1)=1$. Note that all $v_i'$s ($i \neq 1$) are adjacent to 
exactly three vertices. Therefore there exists a vertex $v_j$ such that 
$d(v_i,v_j)=2$, where the existence of $v_j$ follows from 
the assumption $n\geq5$. Hence, $e(v_i)=2$ for $i=2,3,\cdots,n$. Let 
$\mathbf{d}=(0,1, 2,2,\cdots,2,1)^{\prime}$ and 		
$\mathbf{u}=(0,0,2,2,\cdots,2,0)^{\prime}$ be the vectors fixed in  
$\mathbb{R}^{n-1}$. 
Then 
the
distance matrix $	
D(W_n)$ and the 
eccentricity matrix $E(W_n)$ of the wheel graph are $n \times n$ 
symmetric matrices which can be 
written in the 
block form 
\begin{equation}\label{eccentricity matrix of wheel graph}
	D(W_n) = \begin{bmatrix}
		0 &\mathbf{e^{\prime}}\\[3pt]
		\mathbf{e}&\widetilde{D}\\
	\end{bmatrix} ~\text{and} ~~~ 	E(W_n) = \begin{bmatrix}
		0 &\mathbf{e^{\prime}}\\[3pt]
		\mathbf{e}& \widetilde{E}\\
	\end{bmatrix},
\end{equation}
respectively, where $\widetilde{D}= \cir(\mathbf{d}^{\prime}$) and 
$\widetilde{E}=\cir(\mathbf{u^{\prime}}$) are circulant matrices of order 
$n-1$. 

\subsection{The irreducible property of $E(W_n)$}	
Let us recall that 
a square matrix $A$ of order $n$ is said to be \textit{irreducible} if 
there is no
permutation matrix $P$ of order $n$ such that 
\begin{equation*}
	P^{\prime}AP=
	\begin{bmatrix}
		A_1 & A_2\\
		0&A_3
	\end{bmatrix}, 
\end{equation*} where $A_1$ and $A_3$ are square 
matrices of order at least one. An 
equivalent condition for the irreducibility of a non-negative matrix  
is given below.
\begin{theorem}[\hspace{1sp}\cite{Horn}]\label{irreducible}
	Let $A$ be a square matrix of order $n$ whose entries are 
	non-negative. Then $A$ 
	is irreducible if and only if $(I_n+A)^{n-1}>0$.
\end{theorem}
The distance matrix of a connected graph is always irreducible but the	
eccentricity matrix is not irreducible  \cite{Wang1}. Characterizing the 
irreducible eccentricity matrices remains an open problem. It has been 
shown 
that the 
eccentricity matrix of a tree is irreducible \cite{Wang1} and this result 
was extended 
in\cite{Wang2}.

In the following proposition, we show that the eccentricity matrix of the 
wheel graph is irreducible.
\begin{proposition}
	The eccentricity matrix	of $W_n$ is irreducible.
\end{proposition}
\begin{proof}
	For $n=4$, we have $E(W_n)=D(W_n)$ which  is irreducible. If $n\geq 
	5$, then 
	\begin{eqnarray*}
		(I_n+E(W_n))^2=
		\begin{bmatrix}
			n &2(n-3)\mathbf{e^{\prime}}\\[3pt]
			2(n-3)	\mathbf{e}& 
			J_{n-1}+\widetilde{E}^2+2\widetilde{E}+I_{n-1}
		\end{bmatrix}.
	\end{eqnarray*}
	Since $\widetilde{E} \geq 0$, each entry in 
	$J_{n-1}+\widetilde{E}^2+2\widetilde{E}+I_{n-1}$ 
	is at least one. Now, the proof follows from 
	Theorem \ref{irreducible}.
\end{proof}
\subsection{Euclidean distance matrix} We investigate the concept of the
Euclidean 
distance matrix for $E(W_n)$.		A real square matrix $M=(m_{ij})$ of 
order $n$ is said to be an 
\textit{Euclidean 
	distance matrix} (EDM) if there exist $n$ vectors $\mathbf{x_1}, 
\cdots, 
\mathbf{x_n}$ in $\mathbb{R}^r$ such that 
$m_{ij}=\|\mathbf{x_i}-\mathbf{x_j}\|^2_2$, where $\|.\|_2$ stands for 
the	Euclidean norm on $\mathbb{R}^r$ \cite{EDM}. It was shown that 
$D(W_n)$ 
is an EDM	
\cite{EDM}. It is natural to ask whether $E(W_n)$ is an EDM. The 
answer to this question is negative. We show that $E(W_n)$ is not an 
EDM for all $n\geq 5$. Note that $E(W_4)$ is an EDM because $E(W_4) = 
D(W_4)$. We use the following characterization of EDM to prove our 
result.
\begin{theorem}[see \cite{EDM}]\label{EDM charac 1}
	Let $M$ be a real symmetric matrix of order $n$ whose  diagonal 
	entries are zero. Then $M$ is an EDM 
	if and only if $\mathbf{x}^{\prime}M\mathbf{x}\leq 0$ for all
	$\mathbf{x} \in \mathbb{R}^n$ such that 
	$\mathbf{x}^{\prime}\mathbf{e}=0$.
\end{theorem}

\begin{proposition}\label{EDM prop 1}
	For $n\geq5$, the matrix	$E(W_n)$ is not an Euclidean distance matrix.
\end{proposition}

\begin{proof}
	For each $n$, we find a vector $\mathbf{x}\in \mathbb{R}^n$ such that 
	$\mathbf{x^{\prime}}\mathbf{	e}=0$ and 
	$\mathbf{x^{\prime}}E(W_n)\mathbf{x}>0$.
	By a simple manipulation, the matrix 
	$E(W_n)$ given  in (\ref{eccentricity matrix of wheel graph}),
	can be written as 
	\begin{align}\label{E splitted form}
		E(W_n) = \begin{bmatrix}
			0 &\mathbf{e^{\prime}}\\[3pt]
			\mathbf{e}& 2J_{n-1}\\
		\end{bmatrix}-2\begin{bmatrix}
			0 &\mathbf{0^{\prime}}\\[3pt]
			\mathbf{0}& F
		\end{bmatrix}, ~~~\text{where}~~~ 
		F=\cir\big((1,1,\underbrace{0,\cdots,0}_{\text{(n-4) 
				times}},1)\big).
	\end{align} 
	
	We first assume that $n$ is odd. Consider the vector 
	\begin{equation*}
		\mathbf{y^{\prime}}=(0,\mathbf{\tilde{y}^{\prime}}) \in 
		\mathbb{R}^n,~ \text{where}~
		\mathbf{\tilde{y}}=(1,-1,1,-1,\cdots,1,-1)^{\prime} \in 
		\mathbb{R}^{n-1}.
	\end{equation*}
	Then $\mathbf{e}^{\prime}\mathbf{\tilde{y}}=0$ and 
	$F\mathbf{\tilde{y}}=-\mathbf{\tilde{y}}$. This implies that
	\begin{align}\notag
		E(W_n)\mathbf{y}
		& =\begin{bmatrix}
			\mathbf{e^{\prime}}	\mathbf{\tilde{y}}\\[3pt]
			2J_{n-1}	\mathbf{\tilde{y}}
		\end{bmatrix} -2\begin{bmatrix}
			0\\[3pt]
			F\mathbf{\tilde{y}}
		\end{bmatrix}= 2\begin{bmatrix}
			0 \\\mathbf{\tilde{y}}
		\end{bmatrix}=
		2\mathbf{y}.
	\end{align}
	Further, $\mathbf{y}^{\prime}E(W_n)\mathbf{y}= \label{yey}
	2\mathbf{y}^{\prime}\mathbf{y}=2({n-1})>0.$
	We now consider the case  $n$ is even. Then $n=2m$ where 
	$m \geq 3$. 
	
	{\it{Case{(i)}}}: 	 Assume that $m$ is odd. Consider the vector 
	\begin{equation*}
		\mathbf{x}=(0,1,-1,1,-1,\cdots,1,-1,0,1,-1,\cdots,1,-1)^{\prime} \in 
		\mathbb{R}^n, 
	\end{equation*}
	where zero occurs in the first and $(m+1)$-th 
	coordinates. Let $\mathbf{x^{\prime}}=(0,\mathbf{\tilde{x}^{\prime}})$. 
	Then 
	$\mathbf{{e}^{\prime}_{n \times 
			1}}\mathbf{x}=\mathbf{{e}^{\prime}_{(n-1)\times 
			1}}\mathbf{\tilde{x}}=0$. Therefore,
		$	E(W_n)\mathbf{x}
		=-2\begin{bmatrix}
			0&
			(F\mathbf{\tilde{x}})^{\prime}
		\end{bmatrix}^{\prime}.$
	It can be verified that 
	$E(W_n)\mathbf{x}=-2(0,-1,1,-1,1,\cdots,-1,1,-1,0,0,0,1,-1,\cdots,1,-1,1)
	^{\prime}$ where zero occurs in the first, $m$-th, 
	$(m+1)$-th and $(m+2)$-th coordinates. Note that the signs of non-zero 
	coordinates of 
	$E(W_n)\mathbf{x}$ 
	are the same as those of $\mathbf{x}$. Hence, 
	$\mathbf{x}^{\prime}E(W_n)\mathbf{x}=2(n-4)>0$.
	
	{\it{Case{(ii)}}}: Suppose that $m$ is even. In this case, take 
	\begin{equation*}
		\mathbf{x}=(0,1,-1,1-1,\cdots,1,-1,1,0,-1,1,-1,1,\cdots,-1,1,-1)
		^{\prime} \in \mathbb{R}^n,
	\end{equation*} where zero in the first and $(m+1)$-th 
	coordinates. Proceeding similar to the proof of the previous case, we get  
	$\mathbf{x}^{\prime}E(W_n)\mathbf{x}=2(n-4)>0.$
	Therefore, by  Theorem \ref{EDM charac 1}, the result follows.
\end{proof}
\begin{remark}
	In fact, the proof of the above proposition shows that if $n$ is 
	odd, then $\mathbf{y}$ is an eigenvector of $E(W_n)$ corresponding to the 
	eigenvalue $2$.
\end{remark}

\subsection{Spectral radius of $E(W_n)$}
The \textit{spectral radius} of a square matrix $A$, denoted by 
$\rho(A)$, is the maximum of the moduli of the eigenvalues of $A$. 
We recall 
the 
fact that the matrix $E(W_n)$ is non-negative and irreducible. Hence, by 
Perron-Frobenius Theorem \cite{Horn}, the spectral radius $\rho(E(W_n))$ of 
$E(W_n)$ is an 
eigenvalue of $E(W_n)$ and there exists a non-negative eigenvector $\mathbf{x}$ 
corresponding to $\rho(E(W_n))$. That is, 
$E(W_n)\mathbf{x}=\rho(E(W_n))\mathbf{x}$ and $\mathbf{x} \geq 0$. 

We end this section by computing the spectral radius $\rho(E(W_n))$  and 
a 
non-negative eigenvector corresponding to  $\rho(E(W_n))$. Let us recall 
the concepts of equitable partition and the associated characteristic 
matrix.

Let $A$ be a real symmetric matrix of order $n$ whose rows and columns
are indexed by $X = \{1, 2, \cdots , n\}$. Suppose that 
$\{X_1,X_2\}$ is a partition of $X$ such that the number of elements in 
$X_1$ and $X_2$ are $n_1$ and $n_2$ respectively. Let
$	A = 
\left[	\begin{smallmatrix}
	A_{11}&A_{12}\\
	A_{21}&A_{22}
\end{smallmatrix}\right], $
where each $A_{ij}$ denotes the block submatrix of $A$ of order $n_i 
\times n_j$ formed by rows in 
$X_i$ and the columns in $X_j$. If $q_{ij}$ denotes the average row 
sum of $A_{ij}$  
(the sum of all entries in $A_{ij}$ divided by the number of rows), 
then the matrix $Q = (q_{ij})$ is said to be a \textit{quotient matrix} 
of 
$A$ 
with respect to the given partition $X$. If each row sum of the block 
$A_{ij}$ 
is constant (that is, $A_{ij}\mathbf{e_{n_j}}=q_{ij}\mathbf{e_{n_i}}$, 
for 
every 
$i$ and $j$), then the partition is 
called 
\textit{equitable}. The \textit{characteristic matrix} 
$C= (c_{ij} )$ with respect to the partition $X$ is the $n \times 2$ 
matrix 
such that $c_{ij}= 1$, if $i \in 
X_j$ and $0$ otherwise.

We state a result which gives the relation between $\rho(A)$ and $\rho(Q)$.
\begin{theorem}[\hspace{1sp}\cite{Equitable partitionn}]\label{spectral 
radius}
Let $A$ be a non-negative matrix. If $Q$ is the quotient matrix of $A$ 
with respect to an equitable partition, then the spectral 
radius of 
$A$ and $Q$ are equal.
\end{theorem}
Using the notion of equitable partition, we now determine the spectral radius 
of $E(W_n)$.
\begin{theorem}
The spectral radius of $E(W_n)$ is $(n-4)+\sqrt{n^2-7n+15}.$
\end{theorem}
\begin{proof}
We have 
$	E(W_n) = \begin{bmatrix}
0 &\mathbf{e^{\prime}}\\[3pt]
\mathbf{e}& \widetilde{E}\\
\end{bmatrix}$, where $\widetilde{E}=\cir\big((0,0, 
\underbrace{2,\cdots,2}_{\text{(n-4) 
	times}},0)\big)$. Since $\widetilde{E}$ is a circulant matrix, 
it follows that all the row sums of $\widetilde{E}$ are the 
same 
and equal to $2(n-4)$. Now partition the vertex set of $W_n$ as 
$X=\{\{v_1\},\{v_2,\cdots,v_n\}\}$. Then $X$ is an equitable partition.
So, the quotient matrix of $E(W_n)$ with respect to $X$ is
$	Q(W_n) =\left[ \begin{smallmatrix}
0 &n-1\\[3pt]
1& 2(n-4)\\
\end{smallmatrix}\right]$.
As the eigenvalues of $Q$ are $(n-4)\pm\sqrt{n^2-7n+15}$, the proof 
follows by Theorem \ref{spectral radius}.
\end{proof}
From an eigenvector of the quotient matrix $Q$ corresponding to $\rho(Q)$, we 
can obtain an eigenvector of $E(W_n)$ corresponding to $\rho(E(W_n))$ where the 
precise statement is given below.
\begin{lemma}[\hspace{1sp}\cite{Brouwer}]\label{quotient eigen vector}
Let $C$ be a characteristic matrix of an equitable partition $X$ of 
$A$. 
Let 
$v$ be an eigenvector of the quotient matrix $Q$ with respect to $X$ 
for 
an 
eigenvalue $\lambda$. Then $Cv$ is an eigenvector of $A $	for the 
same 
eigenvalue $\lambda$.
\end{lemma}
\begin{remark}
Let $\rho=(n-4)+\sqrt{n^2-7n+15}$. The characteristic matrix of 
$E(W_n)$ 
with respect to $X$ is of order $n \times 2$, whose first row is 
$\mathbf{e_1}$, and the remaining rows are
$\mathbf{e_2}$, where $\mathbf{e_1}, \mathbf{e_2} \in 
\mathbb{R}^2$.  Note that $(\frac{n-1}{\rho},1)^{\prime}$ is an 
eigenvector of 
$Q(W_n) $ corresponding to $\rho(Q(W_n))$. By 
Lemma $\ref{quotient eigen vector}$, $\big(\frac{1}{\rho}(n-1),	
\underbrace{1,\cdots,1}_{\text{(n-1) 
	times}}\big)^{\prime}$ is an eigenvector of $E(W_n)$	corresponding 
to $\rho(E(W_n))$.
\end{remark}

	\section{Determinant of $E(W_n)$}

A  result due to Graham and Pollak \cite{Graham and Pollak}  
showed that $\det(D(T))=(-1)^{n-1}(n-1)2^{n-2}$, where $T$ is a  
tree  on 
$n$ vertices. It is clear that the formula depends only on the number of 
vertices of $T$. This 
result was generalized 
to the distance matrices of weighted trees in \cite{Bapat2}. 
The determinant of the distance matrix of the wheel graph
$W_n$ was studied in \cite{Wheel related graph}. More precisely, 
$\det(D(W_n))=1-n$ if $n$ is even, and $\det(D(W_n))=0$ if 
$n$ is odd.

The primary aim of this section is 	to establish a formula to find the 
determinant of $E(W_n)$. We show that 
\begin{equation*}
	\det(E(W_n))= 
	\begin{cases}
		0 & \text{if } n\equiv1\Mod3,\\
		2^{n-2}(1-n)& \text{if } n \not\equiv1\Mod3,
	\end{cases}
\end{equation*}
which is given in terms of the number of vertices alone. To prove the result, 
we 
obtain the recurrence relation (\ref{recurrence relation2}) involving the 
determinant of the tridiagonal matrix $T_n(-2,-2,-2)$ and a bordered matrix 
$B_n$ (defined in Lemma \ref{modify Tridiag}).	

We need the following result, which gives the determinant of a tridiagonal 
matrix.
\begin{theorem}[\hspace{1sp}\cite{Zhang}, Thm. 5.5]\label{det tridiagonal}
	Let $T_n{(a,b,c)}$ be a tridiagonal matrix of order $n$, 
	where $a,b$ and $c$ 
	are real numbers. If $a^2 \neq 
	4bc$, then
	\begin{equation*}
		\det\big(T_n({a,b,c})\big)= 
		\frac{\alpha^{n+1}-\beta^{n+1}}{\alpha-\beta}~
		\text{where}~\alpha= \frac{a+\sqrt{a^2-4bc}}{2}~ \text{and}~\beta= 
		\frac{a-\sqrt{a^2-4bc}}{2}.
	\end{equation*}
\end{theorem}
In the next two lemmas, we evaluate the determinants of two matrices which will 
be used to find the $\det(E(W_n)$).
\begin{lemma}\label{our tridiag}
	Let  $T_n=T_n{(-2,-2,-2)}$ be a tridiagonal matrix of order $n$. Then
	\begin{equation*}
		\det(T_n)= 
		\begin{cases}
			2^n& \text{if } n\equiv0\Mod3,\\
			-2^n& \text{if } n\equiv1\Mod3,\\
			0& \text{if } n\equiv2\Mod3.\\
		\end{cases}
	\end{equation*}
\end{lemma}
\begin{proof}
	We have $\det(T_n)= (-2)^n ~\det(T_n{(1,1,1)})$. Using Theorem 
	\ref{det tridiagonal}, we get
	\begin{eqnarray*}
		\det(T_n)&=&  (-2)^n\Bigg( \frac{
			\Big(\frac{1+i\sqrt{3}}{2}\Big)^{n+1}-\Big(\frac{1-i\sqrt{3}}{2}\Big)^{n+1}}
		{i\sqrt{3}}\Bigg).
	\end{eqnarray*}
	Since  $\Big(\frac{1+i\sqrt{3}}{2}\Big)^{n+1} =
	\Big(\cos\frac{\pi}{3}+i\sin\frac{\pi}{3}\Big)^{n+1}=
	\cos(n+1)\frac{\pi}{3}+i\sin(n+1)\frac{\pi}{3}$, we have
	\begin{equation}\label{a}
		\det(T_n)=
		(-2)^n 	\Bigg(\frac{2 \sin \Big({(n+1)  
			}\frac{\pi}{3}\Big)}
		{\sqrt{3}}\Bigg).
	\end{equation}
	Note that 
	\begin{equation}\label{b}
		\sin \Big({(n+1)  }\frac{\pi}{3}\Big)= \begin{cases}
			(-1)^k\frac{\sqrt{3}}{2}& \text{if } n=3k ~ \text{or}~n=3k+1,\\
			0& \text{if } n=3k+2.
		\end{cases}
	\end{equation}
	The result follows by substituting (\ref{b}) in (\ref{a}). 
\end{proof}
Let $\alpha$ be a non-zero number and $i\neq j$. The addition of $\alpha$ times 
row $j$ (column $j$) of a matrix $A$ to row $i$ (column $i$) of $A$ is denoted 
by $R_i\rightarrow R_i+ \alpha R_j$ ($C_i\rightarrow C_i+ \alpha C_j$). Note 
that the effect of this operation does not change the determinant of $A$.
\begin{lemma}\label{modify Tridiag}
	Let $T_n$ be defined as in Lemma $\ref{our tridiag}$. Let  
	$B_n= 
	\begin{bmatrix}
		0 &\mathbf{e^{\prime}}\\[3pt]
		\mathbf{e}& T_{n-1}
	\end{bmatrix} $
	be a square matrix of order $n$. Then
	\begin{equation*}
		\det(B_n)= 
		\begin{cases}
			0& \text{if } n\equiv0\Mod3,\\
			2^{n-2}(\frac{n-1}{3})& \text{if } n\equiv1\Mod3,\\
			-2^{n-2}(\frac{n+1}{3})& \text{if } n\equiv2\Mod3.\\
		\end{cases}
	\end{equation*}
\end{lemma}
\begin{proof}
	We prove the result by induction on $n$. It is easy to see that the 
	determinants of $B_2, B_3$ and $B_4$ are $-1,0$ and $4$ respectively.
	
	Let $n \geq 5$. Suppose that the result is true for all positive integers 
	less than $n$. By performing the row operation $R_3\rightarrow R_3-R_2$
	first and then the column operation $C_3\rightarrow 
	C_3-C_2$ on $B_n$, we get
	\begin{equation*}
		\det(B_n)=
		\left \vert\begin{array}{cccc|cccc}
			0&1&0&&1& 1& \cdots&1\\[3pt] 
			1&-2&0&&0&0&\cdots&0 \\[3pt]
			0&0&0&&-2&0&\cdots&0 \\[3pt]\hline
			{1}&0&-2&\\
			\mathbf{e}&\mathbf{0}&\mathbf{0}&&&&T_{n-3}
		\end{array}\right \vert.
	\end{equation*}
	We first expand $\det(B_n)$ along the second row. Thus,
	\begin{eqnarray*}
		\det(B_n)&=& 
		(-1)
		\left \vert \begin{array}{ccc|cccc}
			1&0&&1& 1& \cdots&1\\[3pt] 
			0&0&&-2&0&\cdots&0 \\[3pt]\hline
			{0}&-2&\\[3pt]
			\mathbf{0}&\mathbf{0}&&&&T_{n-3}
		\end{array}\right \vert
		+(-2)  \left \vert\begin{array}{ccc|cccc}
			0&0&&1& 1& \cdots&1\\[3pt] 
			0&0&&-2&0&\cdots&0 \\[3pt]\hline
		{1}&-2&\\[3pt]
			\mathbf{e}&\mathbf{0}&&&&T_{n-3}
		\end{array}\right \vert.
	\end{eqnarray*}
	Next expanding the above determinants along the second column, and the 
	resulting determinants along the second row, we have
	\begin{eqnarray*}
		\det(B_n)&=& 
		(-1)(2)
		\left \vert\begin{array}{ccc|c}
			1&1&&\mathbf{e}^{\prime}\\[3pt] 
			0&-2&&\mathbf{0}^{\prime} \\\hline
			{0}&-2&\\[3pt]
			\mathbf{0}&\mathbf{0}&&T_{n-4}
		\end{array}\right \vert
		+(-2)(2) \left \vert\begin{array}{ccc|c}
			0&1&&\mathbf{e}^{\prime}\\[3pt] 
			0&-2&&\mathbf{0}^{\prime} \\\hline
		{1}&-2&\\[3pt]
			\mathbf{e}&\mathbf{0}&&T_{n-4}
		\end{array}\right \vert\\[6pt]
		&=& (-2)(-2)
		\left \vert\begin{array}{ccc}
			1&& \mathbf{e}^{\prime}\\[3pt]
			\mathbf{0}&&T_{n-4}
		\end{array}\right \vert 
		+ (-4)(-2) ~\det(B_{n-3}).
	\end{eqnarray*}
	Finally, expand the first determinant along the first column, we obtain the 
	recursive 
	formula
	\begin{equation}\label{recurrence relation}
		\det(B_n)= 4  ~\det(T_{n-4}) + 8  ~\det(B_{n-3}).
	\end{equation}
	
	{\it{Case{(i)}}}: Let $n\equiv0\Mod3$. Then, $n-3\equiv0\Mod3$ and 
	$n-4\equiv2\Mod3$. From Lemma \ref{our tridiag} and by induction 
	hypothesis, it follows that $~\det(T_{n-4})=~\det(B_{n-3})=0$. 
	Hence $~\det(B_n)=0$.
	
	{\it{Case{(ii)}}}:
	Suppose $n\equiv1\Mod3$. Then, $n-3\equiv1\Mod3$ and 
	$n-4\equiv0\Mod3$. By Lemma \ref{our tridiag}, we have
	$~\det(T_{n-4})=2^{n-4}$. Using the 
	recurrence relation (\ref{recurrence relation})
	and the induction hypothesis, we get 
	\begin{eqnarray*}
		\det(B_n) = 4(2^{n-4}) + 
		8\bigg(2^{n-5}\Big(\frac{n-4}{3}\Big)\bigg) =	
		2^{n-2}\Big(\frac{n-1}{3}\Big).
	\end{eqnarray*}
	
	{\it{Case{(iii)}}}:
	If $n\equiv2\Mod3$, then the proof is similar to that 
	of {\it{Case{(ii)}}}.
\end{proof}
Later in Remark \ref{wheel graph remark2}, we will show that $\det(B_n$) is 
equal 
to $\det(W_n-e$), where $W_n-e$ is a graph obtained from $W_n$ by deleting an 
edge $e$ which lies on the cycle.

The main result of this section is the following theorem, which gives the 
determinant of $E(W_n)$ in terms of 
the number of vertices of $W_n$.
\begin{theorem}\label{wheel determinant}
	Let $n\geq 5$.  Then \begin{equation*}
		\det(E(W_n))= 
		\begin{cases}
			2^{n-2}(1-n)& \text{if } n\not\equiv1\Mod3,\\
			0& \text{if } n\equiv1\Mod3.
		\end{cases}
	\end{equation*}
\end{theorem}
\begin{proof}
	From (\ref{eccentricity matrix of wheel 
		graph}), we have $\det(E(W_n))=
	\begin{vmatrix}
		0 &\mathbf{e^{\prime}}\\[3pt]
		\mathbf{e}& \widetilde{E}\\
	\end{vmatrix}$.
	
	For each $ i =2,3,\cdots,n$, perform the row operations $R_i\rightarrow 
	R_i-2R_1$ on $E(W_n)$, we get
	\begin{eqnarray*}
		\det(E(W_n))&=&   \left \vert\begin{array}{ccc|ccccc}
			0&1&&1& 1& \cdots&1&1\\[3pt] 
			1&-2&&-2&0&\cdots&0 &-2\\[3pt]\hline
		{1}&-2&&&& \\[3pt]
			1&0&\\[3pt]
			\vdots&\vdots&&&&T_{n-2}\\
			1&0&&&\\[3pt]
			1&-2&&
		\end{array}\right \vert.
	\end{eqnarray*}
	We first add the third, fourth, $\cdots$, $n$-th rows to the second 
	row and then do the similar operations for columns, we get
	\begin{eqnarray*}
		\det(E(W_n))&=&   	\left \vert\begin{array}{ccc|c}
			0&n-1&&\mathbf{e}^{\prime}\\[3pt] 
			n-1&-6(n-1)&&-6\mathbf{e}^{\prime}\\[3pt]\hline
		{\mathbf{e}}&-6\mathbf{e}&&T_{n-2}
		\end{array}\right \vert. 
	\end{eqnarray*}
	By performing the row operation $R_2\rightarrow R_2+6R_1$ first and then 
	the 
	column operation $C_2\rightarrow C_2+6C_1$, we get
	\begin{eqnarray*}
		\det(E(W_n))&=&   \left \vert\begin{array}{ccc|c}
			0&n-1&&\mathbf{e}^{\prime}\\[3pt] 
			n-1&6(n-1)&&\mathbf{0}^{\prime}\\[3pt]\hline
			{\mathbf{e}}&\mathbf{0}&&T_{n-2}
		\end{array}\right \vert.  
	\end{eqnarray*}
	Expanding the above determinant along the second row, we have
	\begin{eqnarray*}
		\det(E(W_n))&=&  (-1)(n-1) \left \vert\begin{array}{cc}
			n-1&\mathbf{e}^{\prime}\\[3pt]
			\mathbf{0}&T_{n-2}
		\end{array}\right \vert+ 6(n-1)\det(B_{n-1}).
	\end{eqnarray*}
	Then expand the first determinant along the first column, we get the 
	recursive 
	formula
	\begin{eqnarray}\label{recurrence relation2}
		\det(E(W_n))&=&  (-1)(n-1)^2~ \det(T_{n-2})+ 
		6(n-1)\det(B_{n-1}).
	\end{eqnarray}

	{\it{Case{(i)}}}: 	Let $n\equiv0\Mod3$. Then $n-1\equiv2\Mod3$ 
	and $n-2\equiv1\Mod3$. Using Lemmas \ref{our tridiag} and \ref{modify 
		Tridiag} in the above
	recurrence relation (\ref{recurrence relation2}), we get 
	\begin{eqnarray*}
		\det(E(W_n))&=& 
		(-1) 
		(n-1)^2(-2^{n-2})+6(n-1)\Big(-2^{n-3}~\frac{n}{3}\Big)=2^{n-2}(1-n).
	\end{eqnarray*}
	
	{\it{Case{(ii)}}}:
	Suppose $n\equiv1\Mod3$. Then, 
	$n-1\equiv0\Mod3$ 
	and $n-2\equiv2\Mod3$. Again, from Lemmas \ref{our tridiag} and  
	\ref{modify 
		Tridiag}, we have $\det(T_{n-2})=\det(B_{n-1})=0$. Hence, 
	$\det(E(W_n))= 0$.
	
	{\it{Case{(iii)}}}: If $n\equiv2\Mod3$, then the proof is similar to that 
	of 
	{\it{Case{(i)}}}.
\end{proof}
As an immediate consequence of Theorem \ref{wheel determinant}, we have the 
following result.
\begin{theorem}\label{invertible}
	Let $n\geq5$. Then $E(W_n)$ is invertible if and only if $n 
	\not\equiv1\Mod3$.
\end{theorem}

\section{Inertia of $E(W_n)$}
Let us recall that for a real symmetric matrix $A$ of order $n$, the 
\textit{inertia} of 
$A$, denoted by 
$\In(A)$, is the ordered triple $(i_+(A), i_{-}(A), i_0(A))$, where $i_+(A)$, 
$i_-(A)$  and  $i_0(A)$ respectively denote the number of positive, 
negative, and 
zero eigenvalues of $A$ including the multiplicities. It is well known that 
$i_+(A)+i_{-}(A)$ is equal to the rank of $A$.

The inertias of the distance matrices of weighted trees and unicyclic graphs 
(i.e., connected graphs 
containing exactly one cycle) have been studied in 
\cite{Bapat2}. It has been shown in \cite{Wheel related graph} that
\begin{equation*}
	\In(D(W_n))= 
	\begin{cases}
		(1,n-1,0)& \text{if $n$ is even},\\
		(1,n-2,1)& \text{if $n$ is odd}.
	\end{cases}
\end{equation*}
From this result, it is evident that $D(W_n)$ has exactly one positive 
eigenvalue which is the spectral radius because $D(W_n)\geq 0$. The inertias of 
the eccentricity 
matrices of paths and lollipop graphs were investigated in \cite{rajes1}. 

In this section, we compute the inertia of the eccentricity matrix of the 
wheel graph  using the notion of interlacing 
property. We also obtain the determinant and the 
inertia of $E(W_n-e)$, where $W_n-e$ is the subgraph 
obtained from the wheel graph $W_n$ by deleting an edge $e$ which lies on 
the cycle.

We begin with an observation that the
eccentricity matrix $E(W_n)$ of the wheel graph on $n$ vertices is not a 
principal submatrix of $E(W_{n+1})$ as
$\cir(\mathbf{m_1}^{\prime})$ 
is  not a principal submatrix of $\cir(\mathbf{m_2}^{\prime})$, where 
$\mathbf{m_1}=(0,0,2,2,\cdots,2,0)^{\prime}\in \mathbb{R}^{n-1}$ and 
$\mathbf{m_2}=(0,0,2,2,\cdots,2,0)^{\prime}\in \mathbb{R}^{n}.$ 
Interestingly, $E(W_n-e)$ is a
leading principal submatrix of $E(W_{n+1})$. 
Without loss of generality, we assume that $e=v_2v_n$. Then,
\begin{eqnarray*}
	D(W_n-e) =
	\begin{bmatrix}
		0 &\mathbf{e^{\prime}}\\[3pt]
		\mathbf{e}& 2J_{n-1}\\
	\end{bmatrix}-\begin{bmatrix}
		0 &\mathbf{0^{\prime}}\\[3pt]
		\mathbf{0}& T_{n-1}(2,1,1)
	\end{bmatrix}.
\end{eqnarray*}
Note that the eccentricities of the 
vertices with respect to the graphs $W_n$ and $W_n-e$ are the same. So,
\begin{equation*}
	E(W_n-e) = \begin{bmatrix}
		0 &\mathbf{e^{\prime}}\\[3pt]
		\mathbf{e}& 2J_{n-1}\\
	\end{bmatrix}-\begin{bmatrix}
		0 &\mathbf{0^{\prime}}\\[3pt]
		\mathbf{0}& T_{n-1}(2,2,2)
	\end{bmatrix}.
\end{equation*}

As $T_{n-1}(2,2,2)$ is a leading principal 
submatrix of 
$T_{n}(2,2,2)$, it follows that $E(W_n-e)$ is a leading principal submatrix 
of $E(W_{n+1}-e)$.
\begin{remark}
	We observe that \begin{equation*}
		E(W_{n+1}) = \begin{bmatrix}
			0 &\mathbf{e^{\prime}}\\[3pt]
			\mathbf{e}& 2J_n\\
		\end{bmatrix}-\begin{bmatrix}
			0 &\mathbf{0^{\prime}}\\[3pt]
			\mathbf{0}& 2F
		\end{bmatrix},~\text{where}~ 	
		F=\cir\big((1,1,\underbrace{0,\cdots,0}_{\text{(n-3) 
				times}},1)\big).\end{equation*}
	If we delete the last row and the last column of 
	$2F$, the resulting matrix is $T_{n-1}(2,2,2)$. Hence $E(W_n-e)$ is a 
	leading principal 
	submatrix of $E(W_{n+1})$.
\end{remark}
\begin{remark}\label{wheel graph remark2}
	After performing the row operations $R_i\rightarrow 
	R_i-2R_1$ on $E(W_n-e)$ for each $ i =2,3,\cdots,n$,  it is evident that 
	the 
	resultant matrix is just $B_n$ (defined in Theorem $\ref{modify Tridiag}$). 
	Therefore, 
	\begin{equation*}
		\det\big(E(W_{n}-e)\big)=\det(B_n)= 
		\begin{cases}
			0& \text{if } n\equiv0\Mod3,\\
			2^{n-2}(\frac{n-1}{3})& \text{if } n\equiv1\Mod3,\\
			-2^{n-2}(\frac{n+1}{3})& \text{if } n\equiv2\Mod3.\\
		\end{cases}
	\end{equation*}
\end{remark}
We need the following interlacing theorem to prove  the 
inertia of $E(W_n)$.
\begin{theorem}[\hspace{1sp}\cite{Zhang}]\label{Interlacing theorem}
	Let $M=\begin{bmatrix}
		A & B\\[3pt]
		B^{\prime}&C\\
	\end{bmatrix}$ be a symmetric matrix of order $n$, where $A$ is a 
	leading principal submatrix of $M$ of order $m ~(1\leq m \leq n)$. If the 
	eigenvalues of $M$ and $A$ are
	$\lambda_1\geq \lambda_2\geq \cdots \geq\lambda_n$ and $\beta_1\geq 
	\beta_2\geq \cdots \geq
	\beta_m$ respectively, then
	$	\lambda_i\geq \beta_i \geq \lambda_{n-m+i},$ for all $i=1,2,\cdots,m.$
	In particular, when $m=n-1$, we have
	\begin{equation*}
		\lambda_1\geq \beta_1 \geq \lambda_2\geq \beta_2 \geq \lambda_3
		\geq \cdots  \geq \lambda_{n-1}\geq \beta_{n-1} \geq \lambda_n.
	\end{equation*}
	Moreover, $	i_+(M)\geq i_+(A)~ \text{and}~i_-(M)\geq i_-(A).$
\end{theorem}

We first obtain the inertia of $E(W_n-e)$ which will be used
to study the inertia of $E(W_n)$.

\begin{lemma}\label{Inertia W-e}
	Let $n\geq 5$. Then
	\begin{equation*}
		\In\big(E(W_n-e)\big)= 
		\begin{cases}
			(\frac{n}{3},\frac{2n-3}{3},1)& \text{if } n\equiv0\Mod3,\\
			(\frac{n+2}{3},\frac{2n-2}{3},0)& \text{if } n\equiv1\Mod3,\\
			(\frac{n+1}{3},\frac{2n-1}{3},0)& \text{if } n\equiv2\Mod3.\\
		\end{cases}
	\end{equation*}
\end{lemma}
\begin{proof}
	The proof is by induction on the number of vertices $n$. If $n=5$ then, by 
	Remark \ref{wheel graph remark2}, $\det\big(E(W_{5}-e)\big)$ is 
	negative. Therefore, $i_0\big(E(W_{5}-e)\big)=0$. Consider the matrix
		$B=\left[\begin{smallmatrix}
			0&1&1\\
			1&0&0\\
			1&0&0
		\end{smallmatrix}\right].$
	Then the eigenvalues of $B$ are $\beta_1=\sqrt{2}$, 
	$\beta_2=0$ and $\beta_3=-\sqrt{2}$. Let 	$\lambda_1\geq \lambda_2\geq 
	\lambda_3\geq \lambda_4 \geq\lambda_5$ be the eigenvalues of $E(W_{5}-e)$. 
	Note that $B$ is a leading principal submatrix  of $E(W_{5}-e)$. By Theorem 
	\ref{Interlacing theorem}, we have $\lambda_1\geq \beta_1$, $\lambda_2\geq 
	\beta_2 \geq \lambda_4$ and $\beta_3 \geq \lambda_5$. As $\lambda_i'$s are 
	non-zero, we have $\lambda_2>0$ and  $\lambda_4<0$. Since 
	$\det\big(E(W_{5}-e)\big)<0$, $i_-\big(E(W_{5}-e)\big)$ must be 3.
	Thus, the inertia 
	of $E(W_5-e)$ is (2,3,0).
	
	Assume that the result is true for $n-1$. 
	Since $E(W_{n-1}-e)$ is a leading 
	principal submatrix of $E(W_{n}-e)$, we use interlacing theorem 
	to the following three cases.
	
	{\it{Case{(i)}}}:	 Let $n\equiv0\Mod3$. Then $n-1\equiv2\Mod3$. 
	Since 
	$\det\big(E(W_{n}-e)\big)=0$, 
	we 
	have $i_0\big(E(W_{n}-e)\big)\geq 1$. By induction hypothesis, 
		$\In\big(E(W_{n-1}-e)\big)= \big(\frac{n}{3},\frac{2n-3}{3},0\big).$
	Using 
	Theorem 
	\ref{Interlacing 
		theorem}, we get $\In\big(E(W_{n}-e)\big)= 
	\big(\frac{n}{3},\frac{2n-3}{3},1\big)$.
	
	{\it{Case{(ii)}}}: If $n\equiv1\Mod3$, then 
	$\det\big(E(W_{n}-e)\big)>0$. Therefore 
	$i_0\big(E(W_{n}-e)\big)=0$. 
	Notice that, in this case, 
	$\In\big(E(W_{n-1}-e)\big)= 
	\big(\frac{n-1}{3},\frac{2n-5}{3},1\big)$. By 
	Theorem \ref{Interlacing theorem} together with the fact that 
	$i_0\big(E(W_{n}-e)\big)=0$, we get $i_+\big(E(W_{n}-e)\big)\geq 
	\big(\frac{n-1}{3}\big)+1$ 
	and $i_-\big(E(W_{n}-e)\big)\geq \big(\frac{2n-5}{3}\big)+1$. 
	Hence,			
	$\In\big(E(W_{n}-e)\big)= 
	\big(\frac{n+2}{3},\frac{2n-2}{3},0\big)$.
	
	{\it{Case{(iii)}}}: Suppose $n\equiv2\Mod3$. Then 
	$\det\big(E(W_{n}-e)\big)<0$. This implies that 
	$i_0\big(E(W_{n}-e\big)=0$. It 
	follows from 
	induction hypothesis that
	$\In\big(E(W_{n-1}-e)\big)= \big(\frac{n+1}{3},\frac{2n-4}{3},0\big)$. 
	Again, by Theorem \ref{Interlacing theorem}, we have
	$		i_+\big(E(W_{n}-e)\big)\geq \frac{n+1}{3}$ and 		
	$i_-\big(E(W_{n}-e)\big)\geq\frac{2n-4}{3}$.
	Note that  
	$\big(\frac{2n-4}{3}\big)$ is an even integer and 
	$\frac{n+1}{3}+\frac{2n-4}{3}=n-1$. Since 
	$\det\big(E(W_{n}-e)\big)<0$, 
	$i_-\big(E(W_{n}-e)\big)$ must be odd. Thus,
	$\In\big(E(W_{n}-e)\big)= 
	\big(\frac{n+1}{3},\frac{2n-1}{3},0\big)$.
\end{proof}
From Theorem \ref{invertible}, it follows that the rank of $E(W_n)$ is $n$ if 
and only if $n \not \equiv1\Mod3$.  The next result gives the rank of $E(W_n)$ 
if $n \equiv1\Mod3$.

\begin{theorem}\label{rank of wn}
	If $n\equiv1\Mod3$, then the rank of $E(W_n)$ is $n-2$.
\end{theorem}
\begin{proof}
	Since $n \equiv 1 \Mod 3$, we have $n=3k+1$ for some $k \geq 1$.
	Consider the vectors $\mathbf{x}$ and $\mathbf{y}$ in $\mathbb{R}^{n}$ 
	which 
	are given by
	\begin{align}\label{rank vectors}
		\mathbf{x}&= 
		(0,\underbrace{1,0,-1},\underbrace{1,0,-1}, \cdots, 
		\underbrace{1,0,-1})^{\prime}
		\intertext{and}\label{rank vectors1}
		\mathbf{y}&= 
		(0,\underbrace{0,1,-1},\underbrace{0,1,-1}, \cdots, 
		\underbrace{0,1,-1})^{\prime}.
	\end{align}
	By simple verification, we can see that
	$E(W_n)\mathbf{x}=E(W_n)\mathbf{y}=0$. Moreover, $\mathbf{x}$ and 
	$\mathbf{y}$ 
	are linearly independent. 
	Therefore, the dimension of the null space of $E(W_n)$ is at least two and 
	hence the rank of $E(W_n)$ is at most $n-2$. To prove the result, it is 
	enough 
	to find an $(n-2) \times (n-2)$ submatrix of $E(W_n)$ whose determinant is 
	non-zero. As $n \equiv 1 \Mod 3$, $n-2 \equiv 2 \Mod 3$. From Remark 
	\ref{wheel 
		graph remark2}, we have 
	$\det\big(E(W_{n-2}-e)\big)\neq 0$. 
	Since 
	$E(W_{n-2}-e)$ is a leading principal submatrix of $E(W_n)$, the result 
	follows.
\end{proof}
We now proceed to compute the inertia of the eccentricity matrix of the wheel 
graph.
\begin{theorem}\label{Inertia of eccentic wheel graph}
	Let $n\geq 5$.
	The inertia of $E(W_n)$ is
	\begin{equation*}
		\In\big(E(W_n)\big)= 
		\begin{cases}
			(\frac{n+3}{3},\frac{2n-3}{3},0)& \text{if } n\equiv0\Mod3,\\
			(\frac{n-1}{3},\frac{2n-5}{3},2)& \text{if } n\equiv1\Mod3,\\
			(\frac{n+1}{3},\frac{2n-1}{3},0)& \text{if } n\equiv2\Mod3.\\
		\end{cases}
	\end{equation*}
\end{theorem}
\begin{proof}
	If $n=5$, then the proof is similar to that of Lemma	\ref{Inertia W-e}. 
	Let 
	$n\geq 6$. We consider the following three cases.
	
	{\it{Case{(i)}}}: Let $n\equiv0\Mod3$. Then from Theorem \ref{wheel 
		determinant}, $\det\big(E(W_n)\big)<0$. Note that $E(W_{n-1}-e)$ is a 
		leading 
	principal submatrix of $E(W_{n})$. By Theorem \ref{Interlacing theorem} and 
	Lemma \ref{Inertia W-e}, we have
	$i_+\big(E(W_n)\big)\geq i_+\big(E(W_{n-1}-e)\big)=\frac{n}{3},$ {and}
	$i_-\big(E(W_n)\big)\geq i_-\big(E(W_{n-1}-e)\big)=\frac{2n-3}{3}.$ As  
	$\big(\frac{2n-3}{3}\big)$ is an odd number 
	and $\det\big(E(W_n)\big)<0$, we get $i_+(E(W_n))=\frac{n}{3}+1$, the 
	result 
	follows 
	in 
	this 
	case.
	
	{\it{Case{(ii)}}}: Suppose $n\equiv1\Mod3$. By Theorem \ref{rank of wn}, 
	rank 
	of 
	$E(W_n)$ is $n-2$ and hence $	i_0\big(E(W_n)\big)=2$. Again, using 
	Theorem 
	\ref{Interlacing theorem} and Lemma \ref{Inertia W-e}, we get
	$	i_+\big(E(W_n)\big)\geq i_+\big(E(W_{n-1}-e)\big)\geq 
	\frac{n-1}{3},~~~\text{and}~~~
	i_-\big(E(W_n)\big)\geq i_-\big(E(W_{n-1}-e)\big)\geq 
	\frac{2n-5}{3}$. Since the sum of $i_+\big((E(W_n)\big), 
	i_-\big((E(W_n)\big)$ 
	and 
	$i_0\big((E(W_n)\big)$ is equal to 
	$n$, we have
		$\In\big(E(W_n)\big)=\big(\frac{n-1}{3},\frac{2n-5}{3},2\big).$
		
		{\it{Case{(iii)}}}: If $n\equiv2\Mod3$, then the proof goes similar to 
		that of 
		{\it{Case{(i)}}}. 
	\end{proof}
	\begin{remark}
		It has been proved that the distance matrix 
		of $W_n$ has a	unique 
		positive eigenvalue \cite{Wheel related graph}. But 
		this 
		property does not hold for $E(W_n)$ because 
		it has at least two positive eigenvalues for all $n\geq 5$.
	\end{remark}
	\section{Inverse formula for $E(W_n)$}
	Graham and Lov\'{a}sz \cite{Graham} obtained the {formula 
		(\ref{Lovaz}) for the} inverse of the 
	distance matrix 
	$D(T)$ 
	of a tree $T$ which is expressed as the sum of the Laplacian 
	matrix and a rank one matrix. 
	Motivated by this result, a similar inverse formula has been derived for 
	the 
	distance matrices of several graphs, see \cite{Bapat, Bapat2, Bapat Inverse 
	of 
		distance block graph,  helm graph, Zhou}.
	
	Balaji et al. 
	\cite{Balaji} obtained an inverse formula similar to (\ref{Lovaz}) for the 
	distance matrix of $W_n$ when $n$ is even. The formula for $D(W_n)^{-1}$ is 
	given as the sum of a Laplacian-like matrix and a rank one matrix. The 
	objective of this section is to find a similar inverse formula for $E(W_n)$ 
	when $n\not \equiv 1 \Mod 3$. We derive 
	this formula by adapting the common technique of finding a suitable 
	Laplacian-like matrix $\widetilde{L}$  and a vector 
	$\mathbf{w}\in \mathbb{R}^n$  such 
	that 
	\begin{equation}\label{identity}
		E(W_n)\mathbf{w}=\alpha\mathbf{e}~ \text{and}~ 
		\widetilde{L}E(W_n)+2I=\beta\mathbf{w}\mathbf{e}^{\prime},
	\end{equation}
	where $\alpha$ and 
	$\beta$ are real numbers. We start 
	with 
	computing the inverse of the 
	eccentricity 
	matrix of the wheel graph on 
	six vertices.
	\begin{example}
		Consider the eccentricity 
		matrix of $W_6$ which is given by
		\begin{eqnarray*}
			E(W_6) = \left[\begin{array}{cc}
				0&\mathbf{e}^{\prime}\\[3pt]
				\mathbf{e}_{5 \times 1}&\cir((0,0,2,2,0))
			\end{array}\right].
		\end{eqnarray*}
		The Laplacian-like matrix for $E(W_6)$ is 
		\begin{eqnarray*}
			\widetilde{L}(W_6) =  \frac{1}{3} \left[\begin{array}{cc|ccccc}
				5 &&-1&-1&-1&-1&-1\\\hline
				-1&&-1&2&-1&-1&2\\[1pt]
				-1&&2&-1&2&-1&-1\\[1pt]
				-1&&-1&2&-1&2&-1\\[1pt]
				-1&&-1&-1&2&-1&2\\[1pt]
				-1&&2&-1&-1&2&-1
			\end{array}\right].
		\end{eqnarray*}
		
		Set $\mathbf{c_1}^{\prime}=(0,1,0,0,1)$, 
		$\mathbf{c_2}^{\prime}=(0,0,1,1,0)$ 
		and 
		$M=\frac{1}{3}\cir(-6\mathbf{e_1}^{\prime}+2\mathbf{c_1}^{\prime}
		-\mathbf{c_2}^{\prime})$.
		Then $\widetilde{L}(W_6)$ can be rewritten as
		\begin{eqnarray*}
			\widetilde{L}(W_6) =  \frac{5}{3}I_6-
			\frac{1}{3} \begin{bmatrix}
				0& \mathbf{e^{\prime}}\\[3pt]
				\mathbf{e} & \mathbf{0}
			\end{bmatrix}+\begin{bmatrix}
				0& \mathbf{0}\\[3pt]
				\mathbf{0} & M
			\end{bmatrix}.
		\end{eqnarray*}
		If $\mathbf{w}=\frac{1}{6}(1,1,1,1,1,1)^{\prime}$, then by direct 
		verification, 
		we  see that 
		\begin{eqnarray*}\label{w6 inverse}
			E(W_6)^{-1}
			&=&\frac{1}{5}\left[\begin{array}{cc|ccccc}
				-4&&1&1&1&1&1\\\hline
			    1&&1&\frac{-3}{2}&1&1&\frac{-3}{2}\\[3pt]
				1&&\frac{-3}{2}&1&\frac{-3}{2}&1&1\\[3pt]
				1&&1&\frac{-3}{2}&1&\frac{-3}{2}&1\\[3pt]
				1&&1&1&\frac{-3}{2}&1&\frac{-3}{2}\\[3pt]
				1&&\frac{-3}{2}&1&1&\frac{-3}{2}&1
			\end{array}\right]=-\frac{1}{2}\widetilde{L}(W_6)+
			\frac{6}{5}\mathbf{w}\mathbf{w^{\prime}}.
		\end{eqnarray*}
	\end{example}
	The above example shows that the inverse of 
	$E(W_6)$ follows some nice patterns. Particularly, the submatrix of 
	$E(W_6)^{-1}$ 
	obtained by deleting the first row and the first column
	is a circulant matrix defined by the vector 
	$\mathbf{x}^{\prime}=\frac{1}{10}(2,-3,2,2,-3) \in 
	\mathbb{R}^{5}$, which follows symmetry in its last four coordinates.
	That is, the second coordinate is equal to the last coordinate of 
	$\mathbf{x}$, 
	and  the third 
	and fourth coordinates of $\mathbf{x}$ are the
	same. In fact, it 
	is observed that the same kind of pattern 
	is retained for  the matrices $E(W_n)^{-1}$ of higher orders. 
	So, most of the entries in the inverse can be obtained {by 
		finding a 
		vector similar to $\mathbf{x}$, which is given below.}
	
	The following significant vectors are identified from the observations made 
	from numerical examples which play a central role in finding a formula for 
	$E(W_n)^{-1}$.
	\vspace{-.6cm}
	\begin{align}
		\notag
		\intertext{For $n\geq 4$, fix the vectors $\mathbf{c_k} \in 
			\mathbb{R}^{n-1}$, where} \label{special vectors}
		\mathbf{c_k}&=\mathbf{e_{k+1}}+
		\mathbf{e_{n-k}} ~
		\text{where}~
		k
		\in \begin{cases}
			\{1,2,\cdots,\frac{n-2}{2}\}& \text{if $n$ is even}, \\
			\{1,2,\cdots,\frac{n-3}{2}\}& \text{if $n$ is odd}.
		\end{cases}
		\intertext{Consider the case $n\equiv2 \Mod 3$ and $n\geq 8$. Define 
		the 
			vectors}\label{x1}
		\mathbf{x_1}&=(2-n)\mathbf{e_1}+\sum_{k=1}^{\frac{n-2}{6}}
		(\mathbf{c_{3k-2}}-2\mathbf{c_{3k-1}}+
		\mathbf{c_{3k}}) \quad \text{if $n$ is even},
		\intertext{and}\label{x2}
		\mathbf{x_2}&= (2-n)\mathbf{e_1}+\sum_{k=1}^{\frac{n-5}{6}}
		\big(\mathbf{c_{3k}}-2\mathbf{c_{3k-1}}\big)+
		\sum_{k=1}^{\frac{n+1}{6}}\mathbf{c_{3k-2}}-
		2\mathbf{e_{\frac{n+1}{2}}} \quad\text{if $n$ is odd}.
		\intertext{In the case $n\equiv0\Mod 3$ and $n\geq 9$, we define}
		\label{x3}
		\mathbf{x_3}&= (-n)\mathbf{e_1}+\sum_{k=1}^{\frac{n}{6}}
		(2\mathbf{c_{3k-2}}-\mathbf{c_{3k-1}})-
		\sum_{k=1}^{\frac{n}{6}-1}\mathbf{c_{3k}}\quad\text{if $n$ is 
			even},
		\intertext{and}\label{x4}
		\mathbf{x_4}&= (-n)\mathbf{e_1}+\sum_{k=1}^{\frac{n-3}{6}}
		\big(2\mathbf{c_{3k-2}}-\mathbf{c_{3k-1}}-\mathbf{c_{3k}}\big)+
		2\mathbf{e_{\frac{n+1}{2}}} \quad \text{if $n$ is odd}.
	\end{align}
	\begin{lemma}\label{special vector lemma}
		If $\mathbf{x_i}$, $i=1,2,3,4$ are the vectors given in $(\ref{x1})- 
		(\ref{x4})$, then
		\begin{align}\label{x}
			\mathbf{x_1}=\mathbf{x_2}&=	
			(2-n,~\underbrace{1,-2,1},~\cdots,
			~\underbrace{1,-2,1})^{\prime}, \intertext{and}\label{y}	
			\mathbf{x_3}=\mathbf{x_4}&= 	
			(-n,~\underbrace{2,-1,-1},~\cdots,~\underbrace{2,-1,-1},~
			2)^{\prime}.
		\end{align}
	\end{lemma}
	\begin{proof}
		Suppose $ n\equiv2\Mod3$ and $n$ is even. Then $n-2 =6l,$ for some 
		positive 
		integer $l$. For $ 1 \leq k \leq \frac{n-2}{6},$
		\begin{align*}
			\mathbf{c_{3k-2}}-2\mathbf{c_{3k-1}}+\mathbf{c_{3k}}&= 
			(\mathbf{e_{3k-1}}-2\mathbf{e_{3k}}+\mathbf{e_{3k+1}})+
			(\mathbf{e_{n-3k+2}}-2\mathbf{e_{n-3k+1}}+\mathbf{e_{n-3k}}).
		\end{align*}
		Thus
		\begin{align*}
			\sum_{k=1}^{\frac{n-2}{6}}(\mathbf{c_{3k-2}}-2\mathbf{c_{3k-1}}+
			\mathbf{c_{3k}})=&
			\Big[(\mathbf{e_{2}}-2\mathbf{e_{3}}+\mathbf{e_{4}})+
			\cdots + (\mathbf{e_{\frac{n}{2}-2}}-2\mathbf{e_{\frac{n}{2}-1}}+
			\mathbf{e_{\frac{n}{2}}})\Big]
			+\\&\Big[(\mathbf{e_{n-1}}-2\mathbf{e_{n-2}}+\mathbf{e_{n-3}})
			+\cdots+ 
			(\mathbf{e_{\frac{n}{2}+3}}-2\mathbf{e_{\frac{n}{2}+2}}
			+\mathbf{e_{\frac{n}{2}+1}})\Big].
		\end{align*}
		This implies that 
		$\mathbf{x_1}=(2-n,~\underbrace{1,-2,1},~\cdots,
		~\underbrace{1,-2,1})^{\prime}.$
		The remaining cases are proved similarly.
	\end{proof}
	\begin{remark}
		Balaji et al.  \cite{Balaji} derived an inverse formula similar to 
		$	(\ref{Lovaz})$ for $D(W_n)$ when $n$ is even. {They had 
			to  deal only with the case when  $n$ is even and introduced only 
			one special 
			vector.} 
		But in the 
		case of $E(W_n)^{-1}$, we have to 
		deal 
		with four cases as mentioned  in $(\ref{x1})-(\ref{x4})$.
		However,  by Lemma $\ref{special vector lemma}$, we have to consider 
		only 
		two 
		cases $ n\equiv2\Mod3$ and $ n\equiv 0\Mod3$. 
	\end{remark}
	Now we construct a Laplacian-like matrix $\widetilde{L}$ of order $n$ for 
	$W_n$ using the 
	above defined vectors. Later we will show that $\widetilde{L}$ is a 
	symmetric 
	matrix of 
	rank $n-1$ (see Lemma \ref{Inverse L symmetric} and Theorem \ref{L hat row 
		sum}).
	\begin{definition}\label{Laplacian defn}
		Let $n\not \equiv 1 \Mod 3$ and $n \geq 8$. Fix the symbols 
		$\mathbf{\bar{x}}$ and $\mathbf{\bar{y}}$ to denote the 
		vectors given in $(\ref{x})$ and $(\ref{y})$ respectively. 
		Define
		\begin{eqnarray}\label{special Laplacian}
			\widetilde{L}:=\frac{n-1}{3}I_n-\frac{1}{3}
			\begin{bmatrix}
				0 &\mathbf{e^{\prime}}\\[3pt]
				\mathbf{e}&\mathbf{0}\\
			\end{bmatrix}+\begin{bmatrix}
				0 &\mathbf{0}^{\prime}\\[3pt]
				\mathbf{0}&M\\
			\end{bmatrix},
		\end{eqnarray}
		where 
		\begin{eqnarray}\label{M}
			M=\frac{1}{3}\begin{cases}
				\cir(\mathbf{\bar{x}}^{\prime})& \text{if } 
				n\equiv2\Mod3,\\[6pt]
				\cir(\mathbf{\bar{y}}^{\prime})& \text{if}  ~ 			
				n\equiv0\Mod3.	
			\end{cases}
		\end{eqnarray}
	\end{definition}
	Throughout this paper, we deal with a circulant matrix of order $n-1$, 
	which is 
	defined by some vector 
	$\mathbf{x}$ in $\mathbb{R}^{n-1}$. On most occasions, the 
	vector $\mathbf{x}$ follows symmetry in its 
	last $n-2$ coordinates. The precise
	definition of symmetry is given below.
	\begin{definition}[\hspace{1sp}\cite{Balaji}]
		Let $n$ be a positive integer such that $n\geq 4$. A vector 
		$\mathbf{x}=(x_1,x_2 \cdots, 
		x_{n-1})^{\prime}$ in $\mathbb{R}^{n-1}$ is said to
		follow symmetry in its last $n-2$ coordinates if
		\begin{equation*}
			x_i = x_{n+1-i} ~\text{for all}~ i =2, 3,\cdots, n-1.
		\end{equation*}
		Equivalently, $\mathbf{x}$
		follows symmetry in its last $n-2$ coordinates if it has the form
		\begin{eqnarray*}
			\mathbf{x}=\begin{cases}
				(x_1,x_2,x_3,\cdots,x_{\frac{n}{2}}, x_{\frac{n}{2}}, 
				\cdots, x_3,x_2)^{\prime}& \text{if $n$ is even},\\[3pt]
				(x_1,x_2,x_3,\cdots,x_{\frac{n-1}{2}},x_{\frac{n+1}{2}}, 
				x_{\frac{n-1}{2}}, \cdots, x_3,x_2)^{\prime} & \text{if $n$ is 
				odd}.
			\end{cases}
		\end{eqnarray*}
	\end{definition}
	\begin{remark}\label{Linear combination of symmetric vectors}
		Let $\mathbf{x}$ and $\mathbf{y}$ be vectors in $\mathbb{R}^{n-1}$.	If 
		each of
		$\mathbf{x}$ and $\mathbf{y}$ follows symmetry in its last $n-2$ 
		coordinates, then $\alpha\mathbf{x}+\beta\mathbf{y}$ also follows 
		symmetry 
		in its last $n-2$ coordinates, where $\alpha$ and $\beta$ are real 
		numbers.
	\end{remark}
	In general, a circulant matrix $C$ need not be symmetric, but if the 
	defining 
	vector of $C$ follows symmetry in its last $n-2$ coordinates, then 
	$C$ is symmetric, which we show next.
	\begin{lemma}\label{circulant symmetry}
		Let 	$C=\cir(\mathbf{c}^{\prime})$, where $\mathbf{c}\in 
		\mathbb{R}^{n-1}$. If $\mathbf{c}$	follows symmetry in its last $n-2$ 
		coordinates, then $C$ is a symmetric matrix.
	\end{lemma}
	\begin{proof}
		Let $\mathbf{c}=(c_1,c_2,\cdots,c_{n-1})^{\prime}$ and let 
		$C=(a_{ij})$. 
		Then $a_{ii}=c_1$ and $i$-th row of $C$ is
			$	
			T^{i-1}(\mathbf{c}^{\prime})=\big(c_{n-(i-1)},c_{n-(i-2)},\cdots,c_{n-2},
			c_{n-1},c_1,c_2,\cdots,c_{n-i}\big)$, where the operator $T$ is
			defined in 
			(\ref{right shift operator}).
		Assume that $i < j$. Then $j=i+k$ 
		for some positive integer $k$. Now,
		\begin{eqnarray*}
			a_{ij}&=&C_{i*}\mathbf{e_j}=T^{i-1}(\mathbf{c}^{\prime})\mathbf{e_j}
			=T^{i-1}(\mathbf{c}^{\prime})\mathbf{e_{i+k}}=c_{k+1}, 
			~~~\text{and}\\
			a_{ji}&=&C_{j*}\mathbf{e_i}=T^{j-1}(\mathbf{c}^{\prime}   
			)\mathbf{e_i}=c_{n-(j-i)}=c_{n-k}.
		\end{eqnarray*}
		Since $\mathbf{c}$ follows symmetry in its last 
		$n-2$ coordinates, we have  
		$c_{k+1}=c_{n-k}$. 
		Hence the proof.
	\end{proof}
	
	\begin{lemma}\label{Inverse L symmetric}
		The matrices $\widetilde{L}$ and $M$ given in Definition 
		$\ref{Laplacian 
			defn}$ are symmetric. 
	\end{lemma}
	\begin{proof}
		Since each of  $\mathbf{\bar{x}}$ and   $\mathbf{\bar{y}}$  follows 
		symmetry in 
		its last $n-2$ coordinates, we have $M$ is symmetric by the above 
		theorem. 
		Therefore, $\widetilde{L}$ is symmetric.
	\end{proof}
	To obtain the main result of this section, we need the following two 
	lemmas.
	\begin{lemma}\label{Me}
		Let $M$ be defined as in $(\ref{M})$. Then
			$M\mathbf{e}=\frac{1}{3}(2-n)\mathbf{e}.$
	\end{lemma}
	\begin{proof}
		Since $M$ is a circulant matrix, all the row sums of $M$ 
		are equal. Therefore,
		\begin{equation*}
			M\mathbf{e}=\alpha\mathbf{e},~~~ \text{where}~~~
			\alpha=\begin{cases}
				\frac{1}{3}	\mathbf{e^{\prime}}\mathbf{\bar{x}}& \text{if } 
				n\equiv2\Mod3,\\[3pt]
				\frac{1}{3}	\mathbf{e^{\prime}}\mathbf{\bar{y}}& \text{if } 
				n\equiv0\Mod3.	
			\end{cases}
		\end{equation*}
		As $\mathbf{\bar{x}^{\prime}}=(2-n, 1,-2,1,1,-2,1,\cdots,
		1,-2,1) \in \mathbb{R}^{n-1}$, we have 
		$\mathbf{e^{\prime}}\mathbf{\bar{x}}=2-n$. Also, the sum of the 
		coordinates 
		of  $\mathbf{\bar{y}^{\prime}}=(-n,2,-1,-1,2,-1,-1,\cdots,
		2,-1,-1,2)\in \mathbb{R}^{n-1}$, except the first and the last 
		coordinates, is zero. So, 
		we get $\mathbf{e^{\prime}}\mathbf{\bar{y}}=2-n$. Thus, 
		$\alpha=\frac{1}{3}(2-n)$ which gives the desired result.
	\end{proof}
	
	\begin{lemma}\label{Si}
		Let  $n \not \equiv 1 \Mod3$ and $M$ be defined  as in $(\ref{M})$. Then
			$	M\widetilde{E}=\frac{1}{3}\cir(\mathbf{z}^{\prime})$, 
			where $\mathbf{z} =(-4,2,4-2n,4-2n,\cdots,4-2n,2)^{\prime} \in 
			\mathbb{R}^{n-1}.$
	\end{lemma}
	\begin{proof}
		Since $\widetilde{E}=
		\cir(\mathbf{u^{\prime}}$) where 	
		$\mathbf{u^{\prime}}=(0,0,2,\cdots,2,0)\in 
		\mathbb{R}^{n-1}$, the columns of $\widetilde{E}$ can be written as
		\begin{equation*}\label{row}
			\widetilde{E}_{*i}=\begin{cases}
				2(\mathbf{e}-\mathbf{e_1}-\mathbf{e_2}
				-\mathbf{e_{n-1}})& \text{if}~ i=1,\\
				2(\mathbf{e}-\mathbf{e_{i-1}}-\mathbf{e_i}-\mathbf{e_{i+1}})& 
				\text{if}~ 2\leq i \leq n-2,\\
				2(\mathbf{e}-\mathbf{e_1}-\mathbf{e_{n-2}}-\mathbf{e_{n-1}})& 
				\text{if}~ i=n-1.
			\end{cases}
		\end{equation*}
		First, assume that $n\equiv2\Mod3$. Note that $M$ and $\widetilde{E}$ 
		are 
		circulant. Using (\ref{product of circulant 
			property}), we write 
		\begin{equation*}
			M\widetilde{E}=\frac{1}{3} 
			\cir(\mathbf{\bar{x}^{\prime}}\widetilde{E}),~\text{where}~
			\mathbf{\bar{x}}=\big(2-n,\underbrace{1,-2,1},\cdots,
			\underbrace{1,-2,1}\big)^{\prime}
			\in 
			\mathbb{R}^{n-1}.
		\end{equation*} 
		Let 
		$\mathbf{z^{\prime}}=\mathbf{\bar{x}^{\prime}}\widetilde{E}
		=(z_1,\cdots,z_{n-1})$.
		Since $\mathbf{\bar{x}^{\prime}}\mathbf{e}=2-n$, we have
		\begin{eqnarray*}
			z_1=\mathbf{\bar{x}^{\prime}}\widetilde{E}_{*1}=
			2(2-n)-2\big((2-n)+1+1\big)=-4.
		\end{eqnarray*}
		Now,
		$	z_2=\mathbf{\bar{x}^{\prime}}\widetilde{E}_{*2}=
		2\mathbf{\bar{x}^{\prime}}\mathbf{e}-2\big((2-n)+1-2\big)=2
		=\mathbf{\bar{x}^{\prime}}\widetilde{E}_{*(n-1)}= z_{n-1}.$
		If  $3\leq i \leq n-2$, then
		$	z_i=\mathbf{\bar{x}^{\prime}}\widetilde{E}_{*i}=
		2\mathbf{\bar{x}^{\prime}}\mathbf{e}-2(x_{i-1}+x_i+x_{i+1}),~\text{where
		 			$x_i$ is 
			the $i$-th coordinate of $\mathbf{\bar{x}}$ }.$
		Note that starting from the second coordinate, the sum of any three 
		consecutive coordinates of 	$\mathbf{\bar{x}}$ is zero. So, we get
		$	z_i=2\mathbf{\bar{x}^{\prime}}\mathbf{e}=2(2-n)$.
		Thus,  
		$\mathbf{z^{\prime}} 
		=(-4,2,~\underbrace{4-2n,\cdots,4-2n}_{\text{$(n-4)$ 
				times}},~2).$
		Suppose $n\equiv0\Mod3$. Then $M\widetilde{E}=\frac{1}{3} 
		\cir(\mathbf{\bar{y}^{\prime}}\widetilde{E})$. By the same arguments as 
		above, we have  
		$\mathbf{\bar{y}^{\prime}}\widetilde{E}=\mathbf{z^{\prime}}$. Hence the 
		proof.
	\end{proof}
	We are now ready to give a formula for the inverse of the eccentricity 
	matrix 
	of $W_n$ which is similar to the form of (\ref{Lovaz}). That is, the 
	inverse of 
	$E(W_n)$ 
	is expressed as a sum of a symmetric Laplacian-like matrix of rank $n-1$ 
	and a  
	rank one 
	matrix. Hereafter, we denote $E(W_n)$ 
	simply by  $E$. 
	\begin{theorem}\label{inverse formula n is even 0 mod}
		Let $n\geq 8$ and $n \not\equiv 1\Mod 3$. Consider the matrix  
		$\widetilde{L}$ 
		given in $(\ref{special Laplacian})$. Suppose that $ 
		\mathbf{w}=\frac{1}{6}(7-n,1,\cdots,1)^{\prime} \in \mathbb{R}^{n}$. 
		Then
		\begin{equation*}
			E^{-1}=
			-\frac{1}{2}\widetilde{L}+\frac{6}{n-1}\mathbf{w}\mathbf{w^{\prime}}.%
			%
		\end{equation*}
	\end{theorem}
	\begin{proof}
		In order to get  identities in	(\ref{identity}), multiplying 
		$\widetilde{L}$ and $E$ gives
		\begin{eqnarray*}
			\widetilde{L}E=
			\left[\begin{matrix}
				\frac{1-n}{3} &A\\[6pt]
				N&S\\
			\end{matrix}\right],
		\end{eqnarray*}
		where 
		\begin{align}\notag
			A&=  
			\frac{n-1}{3}\mathbf{e^{\prime}}-
			\frac{1}{3}\mathbf{e^{\prime}}\widetilde{E},\\\label{N}
			N&=  \frac{n-1}{3}\mathbf{e}+M \mathbf{e}, 
			\intertext{and}\notag
			S&= 
			\frac{n-1}{3}\widetilde{E}-\frac{1}{3}\mathbf{e}\mathbf{e^{\prime}}+M\widetilde{E}.
		\end{align}
		Note that $\widetilde{E}\mathbf{e}=2(n-4)\mathbf{e}$. Therefore,
		\begin{equation*}
			A= 
			\frac{n-1}{3}\mathbf{e^{\prime}}-\frac{1}{3}(2n-8)\mathbf{e^{\prime}}
			=\frac{7-n}{3}\mathbf{e^{\prime}}.
		\end{equation*}
		Using Lemma \ref{Me} in (\ref{N}), we get,
		\begin{equation*}
			N=\frac{n-1}{3}\mathbf{e}+\frac{2-n}{3}\mathbf{e}=\frac{1}{3}\mathbf{e}.
		\end{equation*}
		%
		Combining Lemma \ref{Si} and (\ref{sum of circulant property}), $S$ can 
		be 
		written as
		\begin{equation*}
			S= \frac{1}{3} \cir 
			\big((n-1)\mathbf{u^{\prime}}-\mathbf{e^{\prime}}
			+\mathbf{z^{\prime}}\big), ~\text{where} ~\mathbf{z} 
			=(-4,2,4-2n,\cdots,4-2n,2)^{\prime}.
		\end{equation*}
		It is easy to verify that 
		$(n-1)\mathbf{u^{\prime}}-\mathbf{e^{\prime}}+\mathbf{z^{\prime}}
		=(-5,1,\cdots,1)$. Therefore,
		\begin{eqnarray*}
			S&=&\frac{1}{3}\cir(\mathbf{v}^{\prime}), ~\text{where}~ 
			\mathbf{v}=(-5,1,\cdots,1)^{\prime} \in \mathbb{R}^{n-1}.
		\end{eqnarray*}
		Hence,
		\begin{eqnarray*}
			\widetilde{L}E= \frac{1}{3}
			\left[\begin{matrix}
				1-n &(7-n)\mathbf{e^{\prime}}\\[6pt]
				\mathbf{e}&\cir(\mathbf{v^{\prime}})
			\end{matrix}\right].
		\end{eqnarray*}
		Since $2\mathbf{e_1^{\prime}}+\frac{1}{3}\mathbf{v^{\prime}}
		=\frac{1}{3}\mathbf{e^{\prime}}$, we have
		\begin{eqnarray*}\label{LE+2I}
			\widetilde{L}E+2I&= \frac{1}{3} \begin{bmatrix}
				7-n &(7-n)\mathbf{e^{\prime}}\\[5pt]
				\mathbf{e}&\cir(\mathbf{e^{\prime}})\\
			\end{bmatrix}=2\mathbf{w}\mathbf{e^{\prime}}.
		\end{eqnarray*}
		As $n \not\equiv 1\Mod 3$, we have $E$ is invertible which follows from 
		Theorem \ref{invertible}. So, from 
		the above identity, we get
		\begin{eqnarray}\label{pre inverse eqn}
			2E^{-1}&=&-\widetilde{L} +2\mathbf{we^{\prime}}E^{-1}.
		\end{eqnarray}
		By an easy to verification, we see that $E\mathbf{w} 
		=\big(\frac{n-1}{6}\big)\mathbf{e}$. Therefore,
		\begin{eqnarray}\label{pre inverse w}
			\mathbf{w} =\Big(\frac{n-1}{6}\Big)E^{-1}\mathbf{e}.
		\end{eqnarray}
		Note that $\widetilde{L}$ and $E$ are symmetric. This implies that 
		$E^{-1}$ is 
		also symmetric. Taking the transpose of  (\ref{pre inverse eqn}), we get
		$	2E^{-1}=-\widetilde{L} 
		+2E^{-1}\mathbf{	e}\mathbf{w^{\prime}}.$
		Using (\ref{pre inverse w}), we get the desired result
		\begin{eqnarray*}
			E^{-1}&=&-\frac{1}{2}\widetilde{L} 
			+\frac{6}{n-1}\mathbf{ww^{\prime}}.
		\end{eqnarray*}
	\end{proof}

	Next, we prove that the matrix $\widetilde{L}$ is a Laplacian-like matrix 
	and 
	of rank $n-1$. Hence $\widetilde{L}$ has some of the properties of 
	the Laplacian matrix.
	\begin{theorem}
		If $\widetilde{L}$ is defined as in $(\ref{special Laplacian})$, then  
		$\widetilde{L}$ is a Laplacian-like matrix.
	\end{theorem}
	\begin{proof}
		We have
		\begin{eqnarray*}
			\widetilde{L}\mathbf{e} &=& \frac{n-1}{3}\mathbf{e}-\frac{1}{3}
			\begin{bmatrix}
				n-1\\[3pt]
				\mathbf{e}_{(n-1) \times 1}
			\end{bmatrix}+\begin{bmatrix}
				0\\[3pt]
				M	\mathbf{e}_{(n-1) \times 1}
			\end{bmatrix}
			=\begin{bmatrix}
				0\\[3pt]
				\big(\frac{n-2}{3}\big)\mathbf{e}_{(n-1) \times 1}+
				M	\mathbf{e}_{(n-1) \times 1}
			\end{bmatrix}.
		\end{eqnarray*}
		Then, $\widetilde{L} \mathbf{e}=0$ which follows from Lemma \ref{Me}. 
		Since $\widetilde{L}$ is symmetric, row sums  of $\widetilde{L}$ are 
		also zero.
	\end{proof}

The proof of the following theorem is similar to that of Theorem $2$ in
\cite{Balaji}, which is given for the sake of completeness.
\begin{theorem}\label{L hat row sum}
	Consider the $n \times n$ matrix $\widetilde{L}$ given in $(\ref{special 
		Laplacian})$. Then the rank of $\widetilde{L}$ is $n-1$.
\end{theorem}
\begin{proof}
	Suppose 
	$\widetilde{L}\mathbf{x}=0$ where $\mathbf{x}$ is a non-zero vector in 
	$\mathbb{R}^n$. Then $\mathbf{x^{\prime}}	\widetilde{L} E=0$.
	Using the relation 
	$\widetilde{L}E+2I=2\mathbf{w}\mathbf{e^{\prime}}$,  
	we get	$\mathbf{x^{\prime}}=(\mathbf{x^{\prime}}\mathbf{w}) 
	\mathbf{e^{\prime}}$ where $\mathbf{w}=\frac{1}{6}(7-n,1,1,\cdots, 
	1)^{\prime}$.
	Thus, $\mathbf{x}$ is a scalar multiple of $\mathbf{e}$. This gives that 
	the 
	dimension of nullspace of $\widetilde{L}$ is at most one. Since 
	$\widetilde{L}\mathbf{e}=0$, we have nullity of $\widetilde{L}$ is one. 
	Hence  
	rank of 
	$\widetilde{L}$ is $n-1$.
\end{proof}

\section{Moore-Penrose Inverse of $E(W_n)$}
The Moore-Penrose Inverses  of the incidence 
matrices of the different graphs have been 
studied in the literature, see \cite{MP inverse, Bapat}. A formula 
for $D(W_n)\ssymbol{2}$ is given by Balaji et al. 
\cite{Balaji odd 
	wheel graph} when $n$ is 
odd  and it is written as the sum of a Laplacian-like matrix and a 
rank one 
matrix, see (\ref{Balaji MP formula}). This result 
motivates us to find a similar Moore-Penrose inverse formula for $E(W_n)$ when 
$n \equiv 1 \Mod 3$. Now, we compute $E(W_7)\ssymbol{2}$ in the following 
example.
\begin{example}
	Consider the eccentricity 
	matrix of the wheel graph on seven vertices. Take 
	$\mathbf{c_1}=(0,1,0,0,0,1)^{\prime}$ 
	and $\mathbf{c_2}=(0,0,1,0,1,0)^{\prime}$, 
	$\mathbf{e_4}=(0,0,0,1,0,0)^{\prime}$ and 
	\begin{eqnarray*}
		P&=&\cir\Big(\frac{-35}{18}\mathbf{e_1}+\frac{11}{36}\mathbf{c_1}-
		\frac{7}{36}\mathbf{c_2}
		+\frac{1}{18}\mathbf{e_4}\Big)
		=\cir\Big(\frac{-35}{18},\frac{11}{36},\frac{-7}{36},\frac{1}{18},
		\frac{-7}{36},\frac{11}{36}\Big).
	\end{eqnarray*}
	The Laplacian-like matrix for $E(W_7)$ is given by
	\begin{eqnarray*}		
		\hat{L}(W_7) =
		\frac{1}{6} \left[\begin{array}{cc}
			2&-2\mathbf{e}\\[3pt]
			-2\mathbf{e}_{6 \times 1}& 
			\cir\Big(\big(\frac{1}{3},\frac{11}{6},\frac{-7}{6},\frac{1}{3},
			\frac{-7}{6},
			\frac{11}{6}\big)\Big)
		\end{array}\right]
		=2I_7-
		\frac{1}{3} \begin{bmatrix}
			0& \mathbf{e}^{\prime}\\[3pt]
			\mathbf{e} & \mathbf{0}
		\end{bmatrix}+\begin{bmatrix}
			0& \mathbf{0}^{\prime}\\[3pt]
			\mathbf{0}& P
		\end{bmatrix}.
	\end{eqnarray*}
	Setting $\mathbf{w}=\frac{1}{6}(0,1,1,1,1,1,1)^{\prime}$, it is easy to 
	verify that
	\begin{eqnarray*}		
		E(W_7)\ssymbol{2} = \frac{1}{6} \left[\begin{array}{cc}
			-1&\mathbf{e}^{\prime}\\[3pt]
			\mathbf{e}&
			\cir\Big(\big(0,\frac{-6}{8},\frac{6}{8},0,\frac{6}{8},
			\frac{-6}{8}\big)\Big)
		\end{array}\right]=-\frac{1}{2}\hat{L}(W_7)+
		\mathbf{w}\mathbf{w^{\prime}}.
	\end{eqnarray*}
\end{example}

Hereafter we assume that $n\geq 10$ and $n \equiv 1\Mod 3$. To establish a 
formula for $E(W_n)\ssymbol{2}$, we fix the following vectors in 
$\mathbb{R}^{n-1}$, which are identified from the numerical 
computations.
\begin{itemize}[leftmargin=*]
	\item[]
	If $n$ is even, then fix $\mathbf{z_1}$, which is given by
	\begin{equation}\label{special vector z1}
		\mathbf{z_1}=
		(2n-n^2)\mathbf{e_1}+
		\sum_{k=1}^{\frac{n-4}{6}}\bigg(\frac{(3n-18k+8)}{2}
		\mathbf{c_{3k-2}} -\frac{(3n-18k+4)}{2}\mathbf{c_{3k-1}}+
		\mathbf{c_{3k}}\bigg)+\mathbf{c_{\frac{n}{2}-1}}.\\[6pt]
	\end{equation}
	\item[]
	Suppose $n$ is odd, then fix $\mathbf{z_2}$ as 
	\begin{equation}\label{special vector z2}
		\mathbf{z_2}=
		(2n-n^2)\mathbf{e_1}+
		\sum_{k=1}^{\frac{n-1}{6}}\bigg(\frac{(3n-18k+8)}{2}
		\mathbf{c_{3k-2}}-\frac{(3n-18k+4)}{2}\mathbf{c_{3k-1}}\bigg)+ 
		\sum_{k=1}^{\frac{n-7}{6}}
		\mathbf{c_{3k}}+\mathbf{e_{\frac{n+1}{2}}},
	\end{equation}
	where the vectors   $\mathbf{c_k}'$s are defined in 
	$(\ref{special 	vectors})$.
\end{itemize}
\begin{remark}
	If $n$ is even, then $n=6k+4$ where $k\geq 1$, and 
	if $n$ is odd, then $n=6k+1$ where $k\geq 2$. So, the vectors 
	$(\ref{special 
		vector z1})$ and $(\ref{special vector z2})$ are well defined.
\end{remark}

In order to obtain the desired formula for $E(W_n)\ssymbol{2}$, we introduce 
the following matrix $\hat{L}$ which is analogues to $\widetilde{L}$ 
given in 
Definition \ref{Laplacian defn} 
for the invertible case. 
Similar to the properties of $\widetilde{L}$, we will show that $\hat{L}$ is a 
symmetric
Laplacian-like matrix (see Lemma \ref{L hat symmetric} and Theorem \ref{row 
	sum of L}).
\begin{definition}\label{Pi}
	Let $n \equiv 1 \Mod 3$ and $n 
	\geq 10$.  
	Define
	\begin{eqnarray}\label{special Laplacian Moore}
		\hat{L}:=\frac{n-1}{3}I-\frac{1}{3}
		\begin{bmatrix}
			0 &\mathbf{e^{\prime}}\\[3pt]
			\mathbf{e}&\mathbf{0}\\
		\end{bmatrix}+\begin{bmatrix}
			0 &\mathbf{0}^{\prime}\\[3pt]
			\mathbf{0}&P\\
		\end{bmatrix},
	\end{eqnarray}
	where 
	\begin{eqnarray}\label{P}
		P=	\frac{1}{3(n-1)}\begin{cases}
			\cir(\mathbf{z_1}^{\prime})& \text{if $n$ is even},\\[6pt]
			\cir(\mathbf{z_2}^{\prime})& \text{if $n$ is odd}. 	\end{cases}
	\end{eqnarray}
	
\end{definition}
We now show that the matrix $\hat{L}$ is symmetric and row sums of $P$ are 
equal.
\begin{lemma}\label{L hat symmetric}
	Let  $\hat{L}$ and $P$ be the matrices as defined above. Then $P$   and 
	$\hat{L}$ 
	are symmetric.
\end{lemma}
\begin{proof}
	It is easy to see that each $\mathbf{c_k}$ in  $(\ref{special vectors})$ 
	and 
	$\mathbf{	e_{\frac{n+1}{2}}}$, if $n$ is odd, 
	follows symmetry in its last $n-2$ coordinates. Therefore, by Remark 
	\ref{Linear combination of symmetric vectors}, each of $\mathbf{z_1}$ and 
	$\mathbf{z_2}$ follows symmetry in its last $n-2$ coordinates. By Lemma 
	$\ref{circulant symmetry}$, the matrix $P$ and hence $\hat{L}$ are 
	symmetric.
\end{proof}

\begin{lemma}\label{Pie}
	Let $P$ be the matrix defined in $(\ref{P})$. Then
		$	P\mathbf{e}=\frac{1}{3}(2-n)\mathbf{e}.$
	\end{lemma}
	\begin{proof}
		Since $P$ is a circulant matrix, we have
		$	P\mathbf{e}=\alpha\mathbf{e}$, 
		where
		\begin{eqnarray*}
			\alpha	=	\frac{1}{3(n-1)}\begin{cases}
				\mathbf{z_1^{\prime}}\mathbf{e}& \text{if $n$ is even},\\[6pt]
				\mathbf{z_2^{\prime}}\mathbf{e}& \text{if $n$ is odd}. 	
				\end{cases}
		\end{eqnarray*}
		First consider the case when $n$ is even. Let 
		$\mathbf{z_1}=(z_1^{(1)},z_2^{(1)},\cdots,z_{n-1}^{(1)})^{\prime}$ 
		and   $1 
		\leq k \leq 
		\frac{n-4}{6}$.
		Then, 
		$	z_{3k-1}^{(1)}= \frac{3n-18k+8}{2}, ~ z_{3k}^{(1)}= 
		\frac{-3n+18k-4}{2},  ~
		z_{3k+1}^{(1)}=1,~ \text{and}~  z_{\frac{n}{2}}^{(1)}=1.$
		Note that each $\mathbf{c_k}$ follows symmetry in its last $n-2$ 
		coordinates 
		and  has 
		only 
		two non-zero entries which are one.  So, we have
		\begin{eqnarray*}
			\alpha=	\frac{1}{3(n-1)} \mathbf{z_1^{\prime}}\mathbf{e}
			&=&	\frac{1}{3(n-1)}\Bigg((2n-n^2)+2
			\Bigg(\sum_{k=1}^{\frac{n-4}{6}}\Bigg(\bigg(\frac{8-4}{2}\bigg)+
			1\Bigg)\Bigg)+2\Bigg)\\
			&=& \frac{1}{3(n-1)}(3n-n^2-2)= \frac{2-n}{3}.
		\end{eqnarray*}
		If $n$ is odd, then  from the definition of $\mathbf{z_2}$, we 
		have
		\begin{eqnarray*}
			\alpha
			&=& \frac{1}{3(n-1)}\Bigg((2n-n^2)+2\Bigg(
			\sum_{k=1}^{\frac{n-1}{6}}\bigg(\frac{8-4}{2}\bigg)
			+ \sum_{k=1}^{\frac{n-7}{6}}1\Bigg)+1\Bigg)
			=\frac{2-n}{3}.
		\end{eqnarray*}
	\end{proof}
	In the following, we  
	present two lemmas which will be used to compute $\hat{L}E$. {This 
		helps in establishing identities (\ref{identity}) to achieve a formula 
		for 
		$E\ssymbol{2}$.}
	\begin{lemma}\label{PiV}
		Let $\mathbf{\tilde{v}^{\prime}}= 
		(2,-1,-1, 2,-1,-1,\cdots,2,-1,-1)\in \mathbb{R}^{n-1}$. Let 
		$V=\cir(\mathbf{\tilde{v}^{\prime}})$ and $P$ be defined as in $ 
		(\ref{P})$. 
		Then 
		$PV=\frac{1}{3}(1-n)V.$
	\end{lemma}
	\begin{proof}
		Consider  the operator $T$ defined in (\ref{right shift operator}). Then
		$T(\mathbf{\tilde{v}})+T^2(\mathbf{\tilde{v}})=-\mathbf{\tilde{v}}$ and 
		\begin{equation}\label{T(v)}
			V_{(k+1)*}=	T^k(\mathbf{\tilde{v}}) =
			\begin{cases}
				\mathbf{\tilde{v}}& \text{if } k\equiv0\Mod3,\\
				T(\mathbf{\tilde{v}})& \text{if } k\equiv1\Mod3,\\
				T^2(\mathbf{\tilde{v}})& \text{if } k\equiv2\Mod3.\\
			\end{cases}
		\end{equation}
		Assume that $n$ is even. 	Let  $1 \leq k \leq \frac{n-4}{6}$. Since 
		$\mathbf{c_k}=\mathbf{e_{k+1}}+\mathbf{e_{n-k}},$ we have
		\begin{equation*}
			\mathbf{c_{3k-2}^{\prime}}V=\mathbf{e_{3k-1}^{\prime}}V+
			\mathbf{e_{n-3k+2}^{\prime}}V
			= V_{(3k-1)*}+V_{(n-3k+2)*}
			= T^{3k-2}(\mathbf{\tilde{v}})+  T^{n-3k+1}(\mathbf{\tilde{v}}).
		\end{equation*}
		As  $3k-2 \equiv 1 \Mod 3$ and  $n-3k+1 \equiv 2 \Mod 3$, it 
		follows from (\ref{T(v)}) that 
		\begin{eqnarray}\label{ck v 1}
			\mathbf{c_{3k-2}^{\prime}}V= T(\mathbf{\tilde{v}})+ 
			T^2(\mathbf{\tilde{v}}).
		\end{eqnarray}
		Also,
		\begin{eqnarray}\label{ck v 2}
			\mathbf{c_{3k-1}^{\prime}}V= T^2(\mathbf{\tilde{v}})+ 
			T(\mathbf{\tilde{v}}), 
			\quad \text{and}\quad
			\mathbf{c_{3k}^{\prime}}V= 2\mathbf{\tilde{v}}.
		\end{eqnarray}  
		Since $P=\frac{1}{3(n-1)} \cir(\mathbf{z_1^{\prime}}$) and $V$ is a 
		circulant 
		matrix, it 
		follows from (\ref{product of circulant 
			property}) that  
		$PV=\frac{1}{3(n-1)}\cir(\mathbf{z_1^{\prime}}V).$
		From (\ref{special 
			vector z1}), we have  
		\begin{eqnarray*}
			\mathbf{z_1^{\prime}}V=	(2n-n^2)\mathbf{e_1^{\prime}}V+
			\sum_{k=1}^{\frac{n-4}{6}}
			\Bigg(\frac{(3n-18k+8)}{2}\mathbf{c_{3k-2}^{\prime}}V
			-\frac{(3n-18k+4)}{2}\mathbf{c_{3k-1}^{\prime}}V+
			\mathbf{c_{3k}^{\prime}}V\Bigg)+
			\mathbf{c_{\frac{n}{2}-1}^{\prime}}V.
		\end{eqnarray*}
		By the assumption on $n$, we have $n=6l+4$, where $l\geq 1$. Therefore 
		\begin{eqnarray}\label{ck v 3}
			\mathbf{c_{\frac{n}{2}-1}^{\prime}}V= 
			V_{(\frac{n}{2})*}+V_{(\frac{n}{2}+1)*}=T(\mathbf{\tilde{v}})+ 
			T^2(\mathbf{\tilde{v}}).
		\end{eqnarray}
		Using (\ref{ck v 1}), (\ref{ck v 2}) and (\ref{ck v 3}), we get
		\begin{eqnarray*}
			\mathbf{z_1^{\prime}}V&=&(2n-n^2)\mathbf{\tilde{v}}
			+	
			\sum_{k=1}^{\frac{n-4}{6}}\Bigg(
			\bigg(\frac{8-4}{2}\bigg)\big(T(\mathbf{\tilde{v}})+ 
			T^2(\mathbf{\tilde{v}})\big)+2\mathbf{\tilde{v}}\Bigg)+\big(T(\mathbf{\tilde{v}})+
			T^2(\mathbf{\tilde{v}})\big).
		\end{eqnarray*}
		As $T(\mathbf{\tilde{v}})+T^2(\mathbf{\tilde{v}})=-\mathbf{\tilde{v}}$, 
		we get
		$\mathbf{z_1^{\prime}}V=(2n-n^2)\mathbf{\tilde{v}}-
		\mathbf{\tilde{v}}=-(n-1)^2  \mathbf{\tilde{v}}.$
		Thus, 
		$PV=\frac{1}{3}(1-n)\cir(\mathbf{\tilde{v}}).$
		If $n$ is odd, the computation of $PV$ is quite similar to that of even 
		case.
	\end{proof}
	\begin{lemma}\label{PU}
		Let  $P$ be defined as in 
		$(\ref{P})$. Let $ 
		\mathbf{\bar{u}^{\prime}}=(1,1,0,0,\cdots,0,1)^{\prime} \in 
		\mathbb{R}^{n-1}$ and
		$U=\cir(\mathbf{\bar{u}^{\prime}})$. If
		$\mathbf{z^{\prime}}=\big(5n-n^2-10, 
		2n-n^2+2,~\underbrace{3,-6,3},
		\cdots,\underbrace{3,-6,3}, 2n-n^2+2\big) \in \mathbb{R}^{n-1}$, then
		\begin{equation*}
			PU=\frac{1}{3(n-1)}\cir(\mathbf{z^{\prime}}).
		\end{equation*}
	\end{lemma}
	\begin{proof}
		Since $P$ and $U$ are circulant matrices, we write
			$	PU	=	\frac{1}{3(n-1)}\begin{cases}
				\cir(\mathbf{z_1^{\prime}}U)& \text{if $n$ is even},\\[6pt]
				\cir(\mathbf{z_2^{\prime}}U)& \text{if $n$ is odd}, 	
			\end{cases}$
			by (\ref{product of circulant property}). Note that the columns of 
			$U$ are 
			\begin{equation*}
				U_{*i}=\begin{cases}
					\mathbf{e_1}+\mathbf{e_2}+\mathbf{e_{n-1}}& \text{if}~ 
					i=1,\\
					\mathbf{e_{i-1}}+\mathbf{e_i}+\mathbf{e_{i+1}}& 
					\text{if}~ 2\leq i \leq 	n-2,\\
					\mathbf{e_1}+\mathbf{e_{n-2}}+\mathbf{e_{n-1}}& \text{if}~ 
					i=n-1.
				\end{cases}
			\end{equation*}
			Assume that $n$ is even. We need to compute the 
			vector $\mathbf{z_1^{\prime}}U$. Note that the $i$-th coordinate of 
			$\mathbf{z_1^{\prime}}U$ is $\mathbf{z_1^{\prime}}U_{*i}$. As 
			$\mathbf{z_1}$ 
			follows symmetry in its last $n-2$ coordinates, 
			$\mathbf{z_1^{\prime}}\mathbf{e_i}=\mathbf{z_1^{\prime}}\mathbf{e_{n+1-i}}$
			 for 
			$i=2,3,\cdots n-1$. Now,
			\begin{equation*}\label{zu0}
				\mathbf{z_1^{\prime}}U_{*1}=
				\mathbf{z_1^{\prime}}(\mathbf{e_1}+\mathbf{e_2}+\mathbf{e_{n-1}})
				= 
				\mathbf{z_1^{\prime}}\mathbf{e_1}+2\mathbf{z_1^{\prime}}\mathbf{e_2}
				=(2n-n^2)+2\Big(\frac{3n-18+8}{2}\Big)=5n-n^2-10.
			\end{equation*}
			We now claim that $\mathbf{z_1^{\prime}}U$ follows symmetry in its 
			last $n-2$ 
			coordinates. To see this, first show the second and ($n-1)$-th 
			coordinates of 
			$\mathbf{z_1^{\prime}}U$ are equal. Now
			\begin{eqnarray}\label{zu1}
				\mathbf{z_1^{\prime}}U_{*2}=
				\mathbf{z_1^{\prime}}(\mathbf{e_1}+\mathbf{e_2}+\mathbf{e_{3}})=
				\mathbf{z_1^{\prime}}(\mathbf{e_1}+\mathbf{e_{n-2}}+
				\mathbf{e_{n-1}})=\mathbf{z_1^{\prime}}U_{*(n-1)}.
			\end{eqnarray}
			Also,
			\begin{equation*}\label{z1u}
				\mathbf{z_1^{\prime}}U_{*2}=
				(2n-n^2)+\frac{3n-18+8}{2}-\frac{3n-18+4}{2}= 2n-n^2+2.
			\end{equation*}
			If we assume $3 \leq i \leq n-2$, then
			\begin{eqnarray}\label{zu2}
				\mathbf{z_1^{\prime}}U_{*i}=
				\mathbf{z_1^{\prime}}(\mathbf{e_{i-1}}+\mathbf{e_i}+\mathbf{e_{i+1}})
				=  \mathbf{z_1^{\prime}}(\mathbf{e_{n-i+2}}+
				\mathbf{e_{n-i+1}}+\mathbf{e_{n-i}})= 
				\mathbf{z_1^{\prime}}U_{*(n-i+1)}.
			\end{eqnarray}
			From (\ref{zu1}) and (\ref{zu2}), the vector
			$\mathbf{z_1^{\prime}}U$ follows symmetry in its last $n-2$ 
			coordinates.
			Therefore, to find  $\mathbf{z_1^{\prime}}U,$ it is enough to 
			compute 
			the first $\frac{n}{2}$ coordinates of 
			$\mathbf{z_1^{\prime}}U$. 
			
			Let $3\leq i \leq \frac{n}{2}-2$. Then we have the 
			following three cases.
			
			{\it{Case{(i)}}}:	 Suppose $i=3k$ form some $k\geq1$. Then $1 
			\leq k \leq 
			\frac{n-4}{6}$. Now 
			\begin{eqnarray*}
				\mathbf{z_1^{\prime}}U_{*i}=
				\mathbf{z_1^{\prime}}(\mathbf{e_{i-1}}+\mathbf{e_i}+\mathbf{e_{i+1}})
				&=&	
				\mathbf{z_1^{\prime}}(\mathbf{e_{3k-1}}+\mathbf{e_{3k}}+\mathbf{e_{3k+1}})\\
				&=&	\bigg(\frac{3n-18k+8}{2}\bigg)-
				\bigg(\frac{3n-18k+4}{2}\bigg)+1=3.
			\end{eqnarray*}
			
			{\it{Case{(ii)}}}:  If $i=3k+1$, then 
			\begin{equation*}
				\mathbf{z_1^{\prime}}U_{*i}=
				\mathbf{z_1^{\prime}}(\mathbf{e_{3k}}+\mathbf{e_{3k+1}}+
				\mathbf{e_{3k+2}})
				=-	\bigg(\frac{3n-18k+4}{2}\bigg)+1+
				\bigg(\frac{3n-18(k+1)+8}{2}\bigg)=-6.
			\end{equation*}

			{\it{Case{(iii)}}}:  Let $i=3k+2$. In this case, 
			\begin{equation*}		
				\mathbf{z_1^{\prime}}U_{*i}=
				\mathbf{z_1^{\prime}}(\mathbf{e_{3k+1}}+\mathbf{e_{3k+2}}+
				\mathbf{e_{3k+3}})
				=		1+\bigg(\frac{3n-18(k+1)+8}{2}\bigg)-
				\bigg(\frac{3n-18(k+1)+4}{2}\bigg)=3.
			\end{equation*}
			When $k=\frac{n-4}{6}$, we have 
			$\mathbf{c_{3k}}=\mathbf{c_{\frac{n}{2}-2}}=
			\mathbf{e_{\frac{n}{2}-1}}+\mathbf{e_{\frac{n}{2}+2}}$. Since $n 
			\equiv1\Mod3$ 
			and $n$ is even, $\frac{n}{2}\equiv 2 \Mod3$. If $i=\frac{n}{2}-1$ 
			then 
			\begin{equation*}\label{z3u}
				\mathbf{z_1^{\prime}}U_{*i}= \mathbf{z_1^{\prime}}
				\big(\mathbf{e_{\frac{n}{2}-2}} +\mathbf{e_{\frac{n}{2}-1}}+
				\mathbf{e_{\frac{n}{2}}}\big)
				= -	\bigg(\frac{3n-18(\frac{n-4}{6})+4}{2}\bigg)+1+1=-6.
			\end{equation*}
			Finally,
			$	\mathbf{z_1^{\prime}}U_{*\frac{n}{2}}= \mathbf{z_1^{\prime}}
			\big(\mathbf{e_{\frac{n}{2}-1}} +\mathbf{e_{\frac{n}{2}}}+
			\mathbf{e_{\frac{n}{2}+1}}\big)=2\mathbf{z_1^{\prime}}
			\mathbf{e_{\frac{n}{2}-1}}+\mathbf{z_1^{\prime}}
			\mathbf{e_{\frac{n}{2}}}
			= 3.$
			Thus, for $3\leq i \leq \frac{n}{2}$, 
			\begin{equation*}
				\mathbf{z_1^{\prime}}U_{*i}=\begin{cases}
					3& \text{if}~ i 
					\not \equiv1\Mod3,\\
					-6& \text{if}~ i 
					\equiv1\Mod3.
				\end{cases}
			\end{equation*}
			Hence, 
			$	\mathbf{z_1^{\prime}}U=\big(5n-n^2-10, 
			2n-n^2+2,~\underbrace{3,-6,3},\underbrace{3,-6,3}
			\cdots,\underbrace{3,-6,3}, 2n-n^2+2\big)=\mathbf{z}^{\prime}.$
			If $n$ is odd, we have to calculate $\mathbf{z_2^{\prime}}U$. 
			Computing the 
			vector $\mathbf{z_2^{\prime}}U$ is similar to that of 
			$\mathbf{z_1^{\prime}}U$  
			and it can be verified that $\mathbf{z_2^{\prime}}U = 
			\mathbf{z}^{\prime}$.
		\end{proof}
		\begin{remark}
			While deriving  the Moore-Penrose inverse of $D(W_n)$, the matrix 
			$\bar{L}D(W_n)$ is computed in  \cite{Balaji odd wheel graph} for 
			some 
			Laplacian-like matrix $\bar{L}$. Deleting the first row and the 
			first column 
			of $\bar{L}D(W_n)$ gives a circulant matrix defined by a vector. 
			The vector 
			is computed coordinate-wise where the calculations are lengthy. In 
			our case, 
			by rewriting $E(W_n)$ as in $(\ref{E splitted form})$, the 
			particular 
			type of vector 	is found directly rather than by coordinate-wise 
			computation as in \cite{Balaji odd wheel graph}.
		\end{remark}
		The following lemma is essential for the proof of Theorem \ref{MP 
		formula}.
		\begin{lemma}\label{Moore L}
			Let $\hat{L}$ be defined 
			as
			in $(\ref{special Laplacian Moore})$. Then
			$	\hat{L}E= \frac{1}{3}\begin{bmatrix}
				1-n &(7-n)\mathbf{e^{\prime}}\\[3pt]
				\mathbf{e}& \cir(\mathbf{v^{\prime}})
			\end{bmatrix},$
			where
			\begin{eqnarray*}
				\mathbf{v^{\prime}}=\frac{1}{n-1}\big(17-5n,~\underbrace{n-7,n-7,n+11},
				\cdots,\underbrace{n-7,n-7,n+11}, n-7,n-7\big) \in 
				\mathbb{R}^{n-1}.
			\end{eqnarray*}
		\end{lemma}
		\begin{proof}
			Let $ 
			\mathbf{\bar{u}^{\prime}}=(1,1,0,0,\cdots,0,1)\in 
			\mathbb{R}^{n-1}$. Then,
			from (\ref{E splitted form}), we write
			\begin{eqnarray*}
				E(W_n)= \begin{bmatrix}\label{E rewite}
					0 &\mathbf{e^{\prime}}\\[6pt]
					\mathbf{e}& 2J_{n-1}\\
				\end{bmatrix}-2
				\begin{bmatrix}
					0 &\mathbf{0}^{\prime}\\[6pt]
					\mathbf{0}& \cir(\mathbf{\bar{u}^{\prime}})
				\end{bmatrix}.
			\end{eqnarray*}
			To find $\hat{L}E$, we first compute
			\begin{eqnarray}\label{LE*}
				\hat{L}\begin{bmatrix}
					0 &\mathbf{e^{\prime}}\\[6pt]
					\mathbf{e}& 2J_{n-1}\\
				\end{bmatrix}&=\frac{n-1}{3} 
				\begin{bmatrix}
					0 &\mathbf{e^{\prime}}\\[6pt]
					\mathbf{e}& 2J_{n-1}\\
				\end{bmatrix}-\frac{1}{3}
				\begin{bmatrix}
					\mathbf{e^{\prime}}\mathbf{e}&2\mathbf{e^{\prime}}J_{n-1}\\[6pt]
					\mathbf{0}& J_{n-1}\\
				\end{bmatrix}+
				\begin{bmatrix}
					0 &\mathbf{0^{\prime}}\\[6pt]
					P\mathbf{e}& 2PJ_{n-1}\\
				\end{bmatrix}.
			\end{eqnarray}
			Note that
			\begin{eqnarray}\label{LE11}
				\frac{n-1}{3}\mathbf{e^{\prime}}-\frac{2}{3}\mathbf{e^{\prime}}J_{n-1}
				&=& \frac{1}{3}\big((n-1)-2(n-1)\big)\mathbf{e^{\prime}}
				=\frac{1-n}{3}\mathbf{e^{\prime}}.
			\end{eqnarray}
			From Lemma $\ref{Pie}$,
			\begin{eqnarray}\label{LE12}
				\frac{n-1}{3}\mathbf{e}+P\mathbf{e}
				&=& 
				\frac{1}{3}\big((n-1)+(2-n)\big)\mathbf{e}=\frac{1}{3}\mathbf{e}.
			\end{eqnarray}
			As
			$PJ_{n-1}=\frac{1}{3}(2-n)J_{n-1},$
			we have
			\begin{equation}\label{LE13}
				\frac{n-1}{3}(2J_{n-1})-\frac{1}{3}J_{n-1}+2PJ_{n-1}
				= \frac{1}{3}\big(2(n-1)-1+2(2-n)\big)J_{n-1}
				=\frac{1}{3}J_{n-1}.
			\end{equation}
			Substituting (\ref{LE11}) , (\ref{LE12}) and (\ref{LE13}) in the 
			matrix 
			equation (\ref{LE*}), we get
			\begin{eqnarray}\label{LE1}
				\hat{L}\begin{bmatrix}
					0 &\mathbf{e^{\prime}}\\[6pt]
					\mathbf{e}& 2J_{n-1}\\
				\end{bmatrix}= 	\frac{1}{3}
				\begin{bmatrix}
					1-n &(1-n)\mathbf{e^{\prime}}\\[6pt]
					\mathbf{e}& J_{n-1}\\
				\end{bmatrix}.
			\end{eqnarray}
			Let $U=\cir(\mathbf{\bar{u}^{\prime}})$. Now consider
			\begin{equation}\label{LE2}
				\hat{L}\begin{bmatrix}
					0 &\mathbf{0^{\prime}}\\[6pt]
					\mathbf{0}& U\\
				\end{bmatrix}= \frac{n-1}{3}
				\begin{bmatrix}
					0 &\mathbf{0^{\prime}}\\[6pt]
					\mathbf{0}&  U\\
				\end{bmatrix}-\frac{1}{3}\begin{bmatrix}
					0 &\mathbf{e^{\prime}}U\\[6pt]
					\mathbf{0}&  0_{n-1}\\
				\end{bmatrix}+
				\begin{bmatrix}
					0 &\mathbf{0^{\prime}}\\[6pt]
					\mathbf{0}&  
					PU\\
				\end{bmatrix}= \begin{bmatrix}
					0 &-\mathbf{e^{\prime}}\\[6pt]
					\mathbf{0}&  
					\frac{n-1}{3}U+PU
				\end{bmatrix}.
			\end{equation}
			Multiplying $\hat{L}$ and $E$, and by (\ref{LE1}) and (\ref{LE2}), 
			we get
			\begin{equation}
				\hat{L}E= 	\frac{1}{3}
				\begin{bmatrix}
					1-n &(1-n)\mathbf{e^{\prime}}\\[6pt]
					\mathbf{e}& J_{n-1}\\[6pt]
				\end{bmatrix}-2
				\begin{bmatrix}
					0 &-\mathbf{e^{\prime}}\\[6pt]
					\mathbf{0}&  
					\frac{n-1}{3}U+PU\\[6pt]
				\end{bmatrix}
				= \frac{1}{3}\begin{bmatrix}\label{LE}
					1-n &(7-n)\mathbf{e^{\prime}}\\[6pt]
					\mathbf{e}& J_{n-1}-2(n-1)U-6PU\\[6pt]
				\end{bmatrix}.
			\end{equation}
			Let $\mathbf{z}$ be defined as in Lemma $\ref{PU}$. Then a direct 
			computation 
			yields that the vector
			$	\mathbf{e^{\prime}}-2(n-1)\mathbf{\bar{u}^{\prime}}
			-\Big(\frac{2}{n-1}\Big)\mathbf{z^{\prime}}=\mathbf{v^{\prime}}.$
			By Lemma 
			\ref{PU}, we have 
			\begin{eqnarray*}\label{v prime}
				J_{n-1}-2(n-1)U-6PU&=&
				\cir\Big(\mathbf{e^{\prime}}-2(n-1)\mathbf{\bar{u}^{\prime}}
				-\Big(\frac{2}{n-1}\Big)\mathbf{z^{\prime}}\Big)
				= \cir(\mathbf{v^{\prime}}).
			\end{eqnarray*}
			The proof follows by combining the above equation and (\ref{LE}).
		\end{proof}
		Now, we are in a position to give a formula  for the Moore-Penrose 
		inverse of 
		the eccentricity matrix of the wheel graph, which is expressed in the 
		form 
		of 
		(\ref{Lovaz}).
		\begin{theorem}\label{MP formula}
			Let $\hat{L}$ be the matrix given in $(\ref{special Laplacian 
			Moore})$ and  $ 
			\mathbf{w}=\frac{1}{6}(7-n,1,\cdots,1)^{\prime} \in 
			\mathbb{R}^{n}$. Then
			\begin{equation}\label{Formula MPI}
				E\ssymbol{2}=
				-\frac{1}{2}\hat{L}+\frac{6}{n-1}\mathbf{w}\mathbf{w^{\prime}}.%
				%
			\end{equation}
		\end{theorem}
		\begin{proof}
			From Lemma \ref{Moore L}, 
			\begin{eqnarray*}
				\hat{L}E&=& \frac{1}{3}\begin{bmatrix}
					1-n &(7-n)\mathbf{e^{\prime}}\\[3pt]
					\mathbf{e}& \cir(\mathbf{v^{\prime}})\\
				\end{bmatrix}.
			\end{eqnarray*}
			By a simple manipulation, we write
			\begin{eqnarray}\label{LE4}
				\hat{L}E&=&\frac{1}{3}
				\begin{bmatrix}
					7-n &(7-n)\mathbf{e^{\prime}}\\[3pt]
					\mathbf{e}&\cir(\mathbf{\hat{v}^{\prime}})\\
				\end{bmatrix}-2I_n,
			\end{eqnarray}
			where\begin{eqnarray*}
				\mathbf{\hat{v}^{\prime}}=\frac{1}{n-1}\big(\underbrace{n+11,n-7,n-7},
				\cdots,\underbrace{n+11,n-7,n-7}\big).
			\end{eqnarray*}
			Let 	\begin{eqnarray*}
				X=	
				-\frac{1}{2}\hat{L}+\frac{6}{n-1}\mathbf{w}\mathbf{w^{\prime}}.
			\end{eqnarray*}
			To prove that $X$ is the Moore-Penrose 
			generalized inverse of $E$, we first show that $XE$ is symmetric. 
			Now,
			\begin{eqnarray*}
				XE=	
				-	
				\frac{1}{2}\hat{L}E+\frac{6}{n-1}\mathbf{w}\mathbf{w^{\prime}}E.
			\end{eqnarray*}
			Note that 
			$\mathbf{w^{\prime}}E=\frac{1}{6}(n-1)\mathbf{e^{\prime}}$. 
			Applying 
			(\ref{LE4}) in the above equation, we get
			\begin{equation}\label{RE1}
				XE=	I_n-
				\frac{1}{6}
				\begin{bmatrix}
					7-n &(7-n)\mathbf{e^{\prime}}\\[3pt]
					\mathbf{e}&\cir(\mathbf{\hat{v}^{\prime}})\\
				\end{bmatrix}+\mathbf{w}\mathbf{e^{\prime}}= I_n-
				\frac{1}{6}
				\begin{bmatrix}
					0 &\mathbf{0^{\prime}}\\[3pt]
					\mathbf{0}&\cir(\mathbf{\hat{v}^{\prime}}-\mathbf{e^{\prime}})\\
				\end{bmatrix}.
			\end{equation}
			Also,
			\begin{eqnarray*}
				\mathbf{\hat{v}^{\prime}}-\mathbf{e^{\prime}}=
				\frac{6}{n-1}(\underbrace{2,-1,-1,} \cdots, 
				\underbrace{2,-1,-1}).
			\end{eqnarray*} 
			Therefore, (\ref{RE1}) reduces to 
			\begin{eqnarray}\label{RE11}
				XE=	I_n-\frac{1}{n-1}
				\begin{bmatrix}
					0 &\mathbf{0^{\prime}}\\[3pt]
					\mathbf{0}&\cir(\mathbf{\tilde{v}^{\prime}})\\
				\end{bmatrix},
				~~~\text{where}~~~
				\mathbf{\tilde{v}^{\prime}}= 
				(\underbrace{2,-1,-1,} \cdots, \underbrace{2,-1,-1})\in 
				\mathbb{R}^{n-1}.
			\end{eqnarray}
			Let $V=\cir(\mathbf{\tilde{v}^{\prime}})$. Then $V$ is a
			symmetric matrix that readily follows from Lemma \ref{circulant 
				symmetry} together with the fact that $\mathbf{\tilde{v}}$ 
				follows 
			symmetry in its last $n-2$ coordinates.  This gives $XE$ is 
			symmetric. Also, since $X$ and $E$ are 
			symmetric, so is $EX$. To prove $X=E\ssymbol{2}$, 
			we 
			need to show 
			$EXE=E$ and $XEX=X$. 
			
			From 
			(\ref{RE11}), we have
			\begin{eqnarray*}
				EXE=	E-
				\frac{1}{n-1}\begin{bmatrix}
					0 &\mathbf{e^{\prime}}\\[3pt]
					\mathbf{e}&\widetilde{E}
				\end{bmatrix}
				\begin{bmatrix}
					0 &\mathbf{0^{\prime}}\\[3pt]
					\mathbf{0}&V
				\end{bmatrix}.
			\end{eqnarray*}
			Since $\widetilde{E}=\cir(\mathbf{u^{\prime}})$,  we write 
			$\widetilde{E}V=\cir(\mathbf{u^{\prime}}V)$. Therefore,
			\begin{eqnarray*}\label{ERE}
				EXE=	E-
				\frac{1}{n-1}
				\begin{bmatrix}
					0 &\mathbf{e^{\prime}}V\\[3pt]
					\mathbf{0}&\cir(\mathbf{u^{\prime}}V)
				\end{bmatrix}.
			\end{eqnarray*}
			Since $n \equiv 1\Mod 3$, we have
			$\mathbf{e^{\prime}}\mathbf{\tilde{v}}=0$. This implies that
			$\mathbf{e^{\prime}}V=\mathbf{0}^{\prime}$. Also, 
			$\mathbf{u^{\prime}}\mathbf{\tilde{v}}=0$
			and hence $\mathbf{u^{\prime}}V=\mathbf{0}^{\prime}$. Thus, $EXE=E.$
			
			To complete the proof, we claim that $XEX=X$. From (\ref{RE11}), we 
			get
			\begin{eqnarray}\label{RER}
				XEX&=&	X-
				\frac{1}{n-1}
				\begin{bmatrix}
					0 &\mathbf{0}^{\prime}\\[3pt]
					\mathbf{0}&V
				\end{bmatrix}X.
			\end{eqnarray}
			So, consider
			\begin{eqnarray}\label{VR1}
				\begin{bmatrix}
					0 &\mathbf{0}^{\prime}\\[3pt]
					\mathbf{0}&V
				\end{bmatrix}X &=& \frac{-1}{2}
				\begin{bmatrix}
					0 &\mathbf{0}^{\prime}\\[3pt]
					\mathbf{0}&V
				\end{bmatrix}
				\hat{L}+ 	\frac{6}{n-1}
				\begin{bmatrix}
					0 &\mathbf{0}^{\prime}\\[3pt]
					\mathbf{0}&V
				\end{bmatrix}\mathbf{w}\mathbf{w^{\prime}}.
			\end{eqnarray}
			Now,
			\begin{eqnarray*}
				\notag	\begin{bmatrix}
					0 &\mathbf{0}^{\prime}\\[3pt]
					\mathbf{0}&V
				\end{bmatrix}
				\hat{L}&=& \frac{n-1}{3} \begin{bmatrix}
					0 &\mathbf{0}^{\prime}\\[3pt]
					\mathbf{0}&V
				\end{bmatrix}-\frac{1}{3}
				\begin{bmatrix}
					0 &\mathbf{0}^{\prime}\\[3pt]
					V\mathbf{e}&\mathbf{0}\\
				\end{bmatrix}+ \begin{bmatrix}
					0 &\mathbf{0}^{\prime}\\[3pt]
					\mathbf{0}&VP\\
				\end{bmatrix}.
			\end{eqnarray*}
			Since $P$ and $V$ are circulant, we have 
			$VP=PV$  by (\ref{commutative property}). Also, 
			$\mathbf{\tilde{v}^{\prime}}\mathbf{e}=0$ and 
			$V=\cir(\mathbf{\tilde{v}^{\prime}})$, we get
			$V\mathbf{e}=0$. Hence,
			\begin{eqnarray}\label{VL}
				\begin{bmatrix}
					0 &\mathbf{0}^{\prime}\\[3pt]
					\mathbf{0}&V
				\end{bmatrix}
				\hat{L}
				= \begin{bmatrix}
					0 &\mathbf{0}^{\prime}\\[3pt]
					\mathbf{0}&\frac{1}{3}(n-1)V+PV
				\end{bmatrix}=0,
			\end{eqnarray}
			which follows from Lemma \ref{PiV}. Also, 
			\begin{eqnarray}
				\begin{bmatrix}\notag
					0 &\mathbf{0}^{\prime}\\[3pt]
					\mathbf{0}&V
				\end{bmatrix}\mathbf{w}&=&\frac{1}{6} \begin{bmatrix}
					0 &\mathbf{0}^{\prime}\\[3pt]
					\mathbf{0}&V
				\end{bmatrix}
				\begin{bmatrix}
					7-n\\[6pt]
					\mathbf{e}
				\end{bmatrix}=\frac{1}{6}
				\begin{bmatrix}
					0\\[6pt]
					V\mathbf{e}
				\end{bmatrix}=0.
			\end{eqnarray}
			Substituting the above equation and (\ref{VL})  in (\ref{VR1}), it 
			readily 
			follows   that 
			$\left[	\begin{smallmatrix}
				0 &\mathbf{0}^{\prime}\\[3pt]
				\mathbf{0}&V
			\end{smallmatrix}\right] X =0.$
			Using this in (\ref{RER}), we get $XEX=	X$. Thus, $X=E\ssymbol{2}$.
		\end{proof}
		Next, we present some properties of  $\hat{L}$ which are similar to 
		that of $\widetilde{L}$ given in the previous section. We show that 
		$\hat{L}$ 
		is a Laplacian-like matrix and rank($\hat{L})= $ rank($E\ssymbol{2})-1$ 
		which 
		is 
		analogous to the result  rank($\widetilde{L})=$  rank($E^{-1})-1$, 
		proved in 
		Theorem \ref{L hat row sum}.
		\begin{theorem}\label{row sum of L}
			The matrix 	$\hat{L}$, defined in $(\ref{special Laplacian 
			Moore})$, is a 
			Laplacian-like matrix.
		\end{theorem}
		\begin{proof}
			We have 
			\begin{eqnarray*}
				\hat{L}\mathbf{e}=\frac{1}{3}(n-1)\mathbf{e}-\frac{1}{3}
				\begin{bmatrix}
					(n-1)\\[3pt]
					\mathbf{e}_{(n-1)\times 1}\\
				\end{bmatrix}+\begin{bmatrix}
					0 \\[3pt]
					P\mathbf{e}_{(n-1)\times 1}\\
				\end{bmatrix}.
			\end{eqnarray*}
			By Lemma \ref{Pie}, 
			$P\mathbf{e}_{(n-1)\times 
			1}=\frac{1}{3}(2-n)\mathbf{e}_{(n-1)\times 1}.$
			Hence,
			$\hat{L}\mathbf{e}=0$. Since $\hat{L}$ is symmetric, the result 
			follows.
		\end{proof}

In Theorem \ref{rank of wn}, we proved that rank($E(W_n)$) is $n-2$ if 
$n\equiv1\Mod3$. Note that rank($E)= $ rank($E\ssymbol{2})$ (see \cite{Moore 
	inverse book},p.39). From (\ref{Formula MPI}), it is clear 
that $E\ssymbol{2}$ is 
expressed as the sum of two matrices, where the rank of the second matrix 
$\frac{6}{n-1}\mathbf{w}\mathbf{w^{\prime}}$ is one. 

We will prove that 
rank($\hat{L}$) is $n-3$. 
First, in the following lemma, we show that rank($\hat{L}$) is less than or 
equal to $n-3$. To prove the result, we recall a property of the
Moore-Penrose inverse, which states that $N(A\ssymbol{2})=N(A^{\prime})$ (see 
\cite{Moore inverse book},p.63).
\begin{lemma}\label{rank of L Moore Inverse}
	The rank of $\hat{L}$ given in $(\ref{special Laplacian Moore})$ is at most 
	$n-3$.
\end{lemma}
\begin{proof}
	Consider the vectors $\mathbf{x~\text{and}~y}$ given in (\ref{rank 
		vectors}) and (\ref{rank vectors1}) respectively. Then, $E\mathbf{x}=0$ 
		and 
	$E\mathbf{y}=0$.
	We now show that $\mathbf{x~\text{and}~y}$ belong to $N(\hat{L})$. By the 
	result mentioned above, we have $N(E)=N(E\ssymbol{2})$, because $E$ is 
	symmetric. Therefore,
	$E\ssymbol{2}\mathbf{x}=0$ and 
	$E\ssymbol{2}\mathbf{y}=0$. Using 
	Theorem \ref{MP 
		formula}, we 
	get
	\begin{equation*}
		\frac{-1}{2}\hat{L}\mathbf{x}+
		\frac{6}{n-1}\mathbf{w}\mathbf{w^{\prime}}\mathbf{x}=0 ~\text{and}~
		\frac{-1}{2}\hat{L}\mathbf{y}+
		\frac{6}{n-1}\mathbf{w}\mathbf{w^{\prime}}\mathbf{y}=0.
	\end{equation*}
	Note that $\mathbf{w^{\prime}}\mathbf{x}=0$ and 
	$\mathbf{w^{\prime}}\mathbf{y}=0$. So, we have $\hat{L}\mathbf{x}=0$ and 
	$\hat{L}\mathbf{y}=0$. Moreover, 
	$\hat{L}\mathbf{e} = 
	0$ and 
	$\mathbf{e,x}$ and $\mathbf{y}$ 
	are linearly independent. Therefore, the dimension of $N(\hat{L})$ is at 
	least three, and hence the rank of $\hat{L}$ is 
	at most $n-3$.
\end{proof}
In order to show that the rank of $\hat{L}$ is equal to $ n-3$, we find  
matrices $X$ and $C$ of order $n \times (n-3)$ such that $\hat{L}EX=C$ and 
rank($C) = n-3$. To define the matrix $C$, we need the three vectors 
$\mathbf{p,q,r} 
\in \mathbb{R}^{n-3}$ which are given by
\begin{eqnarray*}\label{pqr}
	\mathbf{p^{\prime}}&=&(-1,~\underbrace{-3,0,0},~\underbrace{-3,0,0},
	~\cdots,~\underbrace{-3,0,0}).\\
	\mathbf{q^{\prime}}&=&(-1,~\underbrace{0,-3,0},
	~\underbrace{0,-3,0},\cdots,~\underbrace{0,-3,0}).\\
	\mathbf{r^{\prime}}&=&(-1,~\underbrace{0,0,-3},
	~\underbrace{0,0,-3},\cdots,~\underbrace{0,0,-3}).
\end{eqnarray*}
In the following lemma, we denote the row vector of size $n$ whose entries are 
all one by $\mathbf{e}_{1\times n}$.
\begin{lemma}\label{rank of L Moore}
	Let $S=\cir(\mathbf{s^{\prime}})$, where 
	$\mathbf{s}^{\prime}=(-2,0,0,~\underbrace{-1,0,0},\underbrace{-1,0,0},
	\cdots,\underbrace{-1,0,0})$
	in $\mathbb{R}^{n-4}$. Consider the matrix $\hat{L}$ given in 
	$(\ref{special 
		Laplacian 
		Moore})$. Let $\mathbf{p,q,r}$ be the vectors defined above.
	If
	\begin{eqnarray*}\label{Vl}\notag
		X= 	\frac{1}{2} \begin{bmatrix}
			n-10 &(n-7)\mathbf{e}_{1\times (n-4)}\\[3pt]
			-\mathbf{e}_{(n-4)\times 1}&3S\\[3pt]
			0_{3 \times 1}&0_{3 \times (n-4)}
		\end{bmatrix} ~\text{and}~ 
		C=	 \begin{bmatrix}
			3I_{n-3}\\[1pt]
			\mathbf{p^{\prime}}\\
			\mathbf{q^{\prime}}\\
			\mathbf{r^{\prime}}
		\end{bmatrix},
	\end{eqnarray*}
	then $\hat{L}EX=C$.
\end{lemma}
\begin{proof}
	Recall from (\ref{LE4}), 
	\begin{eqnarray*}
		\hat{L}E&=&\frac{1}{3} 
		\begin{bmatrix}
			7-n &(7-n)\mathbf{e^{\prime}}\\[6pt]
			\mathbf{e}&\cir(\mathbf{\hat{v}^{\prime}})\\
		\end{bmatrix}-2I_n,
	\end{eqnarray*}
	where $\mathbf{\hat{v}^{\prime}}=\frac{1}{n-1} 
	\big(\underbrace{n+11,n-7,n-7},
	\cdots,\underbrace{n+11,n-7,n-7}\big) \in \mathbb{R}^{n-1}.$ To 
	multiply $\hat{L}E$ with 
	$X$, we partition $\hat{L}E$ accordingly. 
	That is,
	\begin{eqnarray*}
		\hat{L}E&=&\frac{1}{3}
		\begin{bmatrix}
			7-n &(7-n) \mathbf{e}_{1 \times (n-4)} & (7-n) \mathbf{e}_{1 
				\times3}\\[6pt]
			\mathbf{e}_{(n-4) \times 1}&\cir(\mathbf{\bar{v_1}^{\prime}}) & 
			R\\[6pt]
			\mathbf{e}_{3 \times 1} & R^{\prime} & Q
		\end{bmatrix}-2
		\begin{bmatrix}
			1&\mathbf{{0}^{\prime}}&\mathbf{{0}^{\prime}}\\[6pt]
			\mathbf{{0}}&I_{n-4}&0_{(n-4)\times 3}\\[6pt]
			\mathbf{{0}}&0_{3\times (n-4)}&I_3
		\end{bmatrix},
	\end{eqnarray*}
	where 
	\begin{eqnarray*}
		\mathbf{\bar{v_1}^{\prime}}=\frac{1}{n-1}\big(\underbrace{n+11,n-7,n-7},
		\cdots,\underbrace{n+11,n-7,n-7}\big) \in \mathbb{R}^{n-4},
	\end{eqnarray*}
	\begin{eqnarray*}
		R = 
		\begin{bmatrix}
			\mathbf{\bar{v_1}}& T(\mathbf{\bar{v_1}}) &
			T^2(\mathbf{\bar{v_1}})
		\end{bmatrix}, ~\text{and}~ Q= \frac{1}{n-1}
		\begin{bmatrix}
			n+11&n-7&n-7\\
			n-7&n+11&n-7\\
			n-7&n-7&n+11
		\end{bmatrix}.
	\end{eqnarray*}
	Then, 
	\begin{eqnarray}\label{LEX}
		\hat{L}EX&=&\frac{1}{6}
		\begin{bmatrix}
			-6(7-n) &A_{12}\\[6pt]
			A_{21} & A_{22}\\[6pt]
			A_{31} & A_{32}
		\end{bmatrix}-\begin{bmatrix}
			n-10 &(n-7)\mathbf{e}_{1\times (n-4)}\\[3pt]
			-\mathbf{e}_{(n-4)\times 1}&3S\\[3pt]
			0_{3 \times 1}&0_{3 \times (n-4)}
		\end{bmatrix},
	\end{eqnarray}
	where
	\begin{eqnarray*}
		A_{12}&=&-(7-n)^2 \mathbf{e}_{1 \times (n-4)}+
		3(7-n) \mathbf{e}_{1 \times (n-4)}S,\\[6pt]
		A_{21}&=& (n-10)\mathbf{e}_{(n-4) \times 1}-
		\cir(\mathbf{\bar{v_1}^{\prime}})\mathbf{e}_{(n-4) \times 1}, \\[6pt]
		A_{22}&=& (n-7)J_{n-4}+\cir(\mathbf{\bar{v_1}^{\prime}})S,\\[6pt]
		A_{31}&=& (n-10)\mathbf{e}_{3 \times 1}- R^{\prime}\mathbf{e}_{(n-4) 
		\times 
			1},\\[6pt]
		A_{32}&=&(n-7)J_{3 \times (n-4)}+3R^{\prime}S.
	\end{eqnarray*}
	As $S$ is a circulant matrix defined by the vector $\mathbf{s}$, we have 
	$\mathbf{e}_{1 \times (n-4)}S=\alpha \mathbf{e}_{1 \times (n-4)}$, 
	where $\alpha$ is the sum of the coordinates of $\mathbf{s}$. Since the 
	non-zero coordinates of $\mathbf{s}$ are $-2$, which occurs at exactly one 
	place, and $-1$, which occurs in  $\big((\frac{n-4}{3})-1\big)$ places,
	\begin{eqnarray*}
		\alpha&=&\mathbf{e}_{1 \times (n-4)}\mathbf{s^{\prime}}= 
		-2+(-1)\Big(\frac{n-7}{3}\Big)
		= \frac{1-n}{3}.
	\end{eqnarray*}
	Therefore,
	\begin{eqnarray}\label{a12}
		A_{12}&=& \big(-(7-n)^2 +(7-n)(1-n)\big) \mathbf{e}_{1 \times (n-4)}
		= 6(n-7) \mathbf{e}_{1 \times (n-4)}.
	\end{eqnarray}
	Now, 
	\begin{equation*}\label{V1E}
		\mathbf{\bar{v_1}^{\prime}}\mathbf{e}_{(n-4) \times 1}
		=\frac{1}{n-1}\bigg(\big[(n+11)+2(n-7)\big]
		\Big(\frac{n-4}{3}\Big)\bigg)=n-4,
	\end{equation*}
	and hence	$\cir\big(\mathbf{\bar{v_1}^{\prime}}\big)\mathbf{e}_{(n-4) 
		\times 
		1}= 
	(n-4)\mathbf{e}_{(n-4) \times 1}$. So, 
	\begin{equation}\label{a21}
		A_{21}=\big( (n-10)-(n-4)\big) \mathbf{e}_{(n-4) \times 1}
		=-6\mathbf{e}_{(n-4) \times 1}.
	\end{equation}
	We have
	\begin{eqnarray}\label{V1S}
		\mathbf{\bar{v_1}^{\prime}}\mathbf{s}
		&=&\frac{1}{n-1}\bigg((n+11)(-2)+(n+11)(-1)\Big(\frac{n-7}{3}\Big)\bigg)
		=\frac{-1}{3}(n+11),
	\end{eqnarray}
	and
	\begin{equation}\label{V2S}
		\big(T(\mathbf{\bar{v_1}})\big)^{\prime}\mathbf{s}
		=\frac{1}{n-1}\bigg((n-7)(-2)+(n-7)(-1)\Big(\frac{n-7}{3}\Big)\bigg)
		=\frac{-1}{3}(n-7)=\big( 
		T^2(\mathbf{\bar{v_1}})\big)^{\prime}\mathbf{s}.
	\end{equation}
	We need to compute the vector $\mathbf{\bar{v_1}^{\prime}}S$.
	Note that $n-4\equiv 0 \Mod3$. 
	Using  Lemma \ref{Image of a vector which has repeated components} to 
	the matrix  $S=\cir(\mathbf{s}^{\prime})$ and the vector  
	$\mathbf{\bar{v_1}}$, 
	it is enough to compute the first 
	three coordinates of $\mathbf{\bar{v_1}^{\prime}}S$. From (\ref{V1S}) 
	and (\ref{V2S}), we have
	\begin{eqnarray}\label{v1S}
		\mathbf{\bar{v_1}^{\prime}}S 
		&=&\frac{-1}{3}(\underbrace{n+11,n-7,n-7},\cdots,\underbrace{n+11,n-7,n-7})
		= \Big(\frac{1-n}{3}\Big)\mathbf{\bar{v_1}^{\prime}}.
	\end{eqnarray}	
	Similarly, we have 
	\begin{eqnarray}\label{v2s}
		\big(T^i(\mathbf{\bar{v_1}})\big)^{\prime}S
		&=&\Big(\frac{1-n}{3}\Big) \big(T^i(\mathbf{\bar{v_1}})\big)^{\prime}, 
		~\text{for}~ 
		1\leq i \leq 
		2.
	\end{eqnarray}	
	Therefore,
		$	\cir(\mathbf{\bar{v_1}^{\prime}})S= \cir 
		(\mathbf{\bar{v_1}^{\prime}}S)=\frac{1}{3}(1-n) 
		\cir(\mathbf{\bar{v_1}^{\prime}}),$
		and hence
		\begin{equation}\label{a22}
			A_{22}=
			(n-7)J_{n-4}+(1-n)\cir(\mathbf{\bar{v_1}^{\prime}}).
		\end{equation}
		Note that the sum of the entries of $\mathbf{\bar{v_1}}, 
		T(\mathbf{\bar{v_1}})$ and 
		$T^2(\mathbf{\bar{v_1}})$ are the same and equal to $(n-4)$. Therefore, 
		we 
		have
		$R^{\prime}\mathbf{e}_{(n-4)\times 1}=\begin{bmatrix}
			\mathbf{\bar{v_1}}&
			T(\mathbf{\bar{v_1}})&
			T^2(\mathbf{\bar{v_1}})
		\end{bmatrix}^{\prime}\mathbf{e}_{(n-4)\times 
		1}=(n-4)\mathbf{e}_{3\times 1}, ~ 
		\text{and}$
		\begin{eqnarray}\label{a31}
			A_{31}&=& (n-10)\mathbf{e}_{3 \times 1}- 
			R^{\prime}\mathbf{e}_{(n-4) \times 1}= 
			-6\mathbf{e}_{3\times 1}.
		\end{eqnarray}
		By (\ref{v1S}) and (\ref{v2s}), we have 
		\begin{eqnarray*}	\notag
			R^{\prime}S =  {\begin{bmatrix}
					\big(	\mathbf{\bar{v_1}^{\prime}}S\big)^{\prime}&
					\Big(	
					\big(T(\mathbf{\bar{v_1}})\big)^{\prime}S\Big)^{\prime}&
					\Big(	
					\big(T^2(\mathbf{\bar{v_1}})\big)^{\prime}S\Big)^{\prime}
				\end{bmatrix}^{\prime}=\Big(\frac{1-n}{3}\Big)R^{\prime}}.
		\end{eqnarray*}
		Consequently,
		\begin{eqnarray}\label{a32}
			A_{32}&=&(n-7)J_{3 \times (n-4)}+3R^{\prime}S= (n-7)J_{3 \times 
				n-4}+(1-n)R^{\prime}.
		\end{eqnarray}
		Substituting (\ref{a12}), (\ref{a21}), (\ref{a22}), (\ref{a31}) and 
		(\ref{a32}) 
		in (\ref{LEX}), we get
		\begin{eqnarray}\label{LEX 1}
			\hat{L}E X
			&=&  \begin{bmatrix}
				3 & 0_{1\times (n-4)}\\[6pt]
				0_{(n-4)\times 1}&
				\frac{1}{6}\Big((n-7)J_{n-4}+(1-n)\cir(\mathbf{\bar{v_1}^{\prime}})\Big)-3S\\[6pt]
				-\mathbf{e}_{3 \times 
					1}&\frac{1}{6}\Big((n-7)J_{3 \times 
					(n-4)}+(1-n)R^{\prime})\Big)
			\end{bmatrix}.
		\end{eqnarray}
		Note that
		\begin{align*}
			\frac{1}{6}\Big((n-7)\mathbf{e}_{1\times 
				(n-4)}+(1-n)\mathbf{\bar{v_1}^{\prime}}\Big)&=
			(\underbrace{-3,0,0},\underbrace{-3,0,0},\cdots,
			\underbrace{-3,0,0}),\\
			\frac{1}{6}\Big((n-7)\mathbf{e}_{1\times 
				(n-4)}+(1-n)\big(T(\mathbf{\bar{v_1}})\big)^{\prime}\Big)
			&=				(\underbrace{0,-3,0},\underbrace{0,-3,0},\cdots,
			\underbrace{0,-3,0}),
			\intertext{and}
			\frac{1}{6}\Big((n-7)\mathbf{e}_{1\times 
				(n-4)}+(1-n)\big(T^2(\mathbf{\bar{v_1}})\big)^{\prime}\Big)
			&= 	(\underbrace{0,0,-3},\underbrace{0,0,-3},
			\cdots,\underbrace{0,0,-3}).
		\end{align*}
		Hence,						
		\begin{align}\label{LEX 	2}				
			\frac{1}{6}\Big((n-7)\mathbf{e}_{1\times 
				(n-4)}+(1-n)\mathbf{\bar{v_1}^{\prime}}\Big)-3\mathbf{s^{\prime}}
			&= 3\mathbf{e_1^{\prime}},
			\intertext{and}
			\begin{bmatrix}
				-\mathbf{e}_{3 \times 
					1} & \frac{1}{6}\Big((n-7)J_{3 \times 
					n-4}+(1-n)R^{\prime})\Big)
			\end{bmatrix}
			&=\label{LEX 3}
			\begin{bmatrix}
				\mathbf{p}&
				\mathbf{q}&
				\mathbf{r}
			\end{bmatrix}^{\prime}.	
		\end{align}
		Substituting (\ref{LEX 2}) and   (\ref{LEX 3}) in (\ref{LEX 1}), it 
		follows that 
		$\hat{L}E X=C$.       
	\end{proof}
	\begin{theorem}\label{rank of L hat}
		Let $\hat{L}$ be the matrix given in $(\ref{special Laplacian Moore})$. 
		Then 
		the rank of $\hat{L}$ is $n-3$.
	\end{theorem}
	\begin{proof}
		From Lemma \ref{rank of L Moore}, we have 
		$	\text{rank}(\hat{L})\geq \text{rank}(C) = n-3.$
		The reverse inequality follows from Lemma \ref{rank of L Moore Inverse}.
	\end{proof}
	\subsection*{Concluding remarks}
	In this article, we studied the eccentricity matrix $E(W_n)$ of the wheel 
	graph 
	$W_n$. We computed the determinant  and the  inertia of $E(W_n)$. For the 
	non-singular case, we established a formula for the inverse of $E(W_n)$ 
	which 
	was expressed as the sum of a symmetric Laplacian-like matrix and a rank 
	one 
	matrix. A similar result was proved for the Moore-Penrose inverse of 
	$E(W_n)$ 
	for the singular case. It would be interesting to study similar 
	inverse formulae for the eccentricity matrices of other graphs.

\end{document}